\documentclass[a4paper,11pt,fleqn]{article}
\usepackage{enumitem}
\usepackage{amsmath}
\usepackage{amssymb}
\usepackage{mathtools}
\usepackage{pifont}
\usepackage{theorem} 
\usepackage{euscript}
\usepackage{exscale,relsize}
\usepackage{epic,eepic}
\usepackage{multicol}
\usepackage{tikz}
\usepackage{circuitikz}
\usetikzlibrary{arrows.meta,arrows}
\usetikzlibrary{angles,quotes,decorations.markings}
\usepackage{pgfplots}
\usepackage{makeidx}
\usepackage{srcltx} 
\usepackage{url}
\usepackage{charter}
\usepackage{mathrsfs} 
\usepackage{graphicx}
\usepackage{authblk}
\usepackage{float}
\newcommand{\email}[1]{\href{mailto:#1}{\nolinkurl{#1}}}
\usepackage[normalem]{ulem} 
\newlength{\mySubFigSize}
\definecolor{labelkey}{rgb}{0,0.08,0.45}
\definecolor{refkey}{rgb}{0,0.6,0.0}
\definecolor{Brown}{rgb}{0.45,0.0,0.05}
\definecolor{dgreen}{rgb}{0.00,0.60,0.00}
\definecolor{dblue}{HTML}{0455BF}
\definecolor{orng}{HTML}{D35400}
\definecolor{dred}{HTML}{D90404}
\definecolor{Dblue}{HTML}{8602DC}
\definecolor{bluep}{HTML}{0171BD}
\definecolor{pblue}{rgb}{0.1176,0.5647,1}
\definecolor{pgreen}{rgb}{0.1961,0.8039,0.1961}
\definecolor{pred}{rgb}{1.0,0.2706,0.0}
\definecolor{pyellow}{rgb}{1.0,0.6471,0.0}

\RequirePackage[colorlinks,hyperindex]{hyperref}
\hypersetup{linktocpage=true,citecolor=dblue,linkcolor=dgreen}
\renewcommand\familydefault{\rmdefault}
\DeclareMathAlphabet{\mathrm}{OT1}{\familydefault}{m}{n}
\makeatletter
\def\operator@font{\mathgroup\symoperators\rm}
\makeatother
\renewcommand{\leq}{\ensuremath{\leqslant}}
\renewcommand{\geq}{\ensuremath{\geqslant}}

\newcommand{\minimize}[2]{\ensuremath{\underset{\substack{{#1}}}%
{\text{\rm minimize}}\;\;#2 }}

\newcommand{\scal}[2]{{\langle{{#1}\mid{#2}}\rangle}}

\newcommand{\menge}[2]{\big\{{#1}~ |~{#2}\big\}}

\newcommand{\HHH}{\ensuremath{\boldsymbol{\mathcal H}}}
\newcommand{\GGG}{\ensuremath{{\boldsymbol{\mathcal G}}}}
\newcommand{\Argmin}{\ensuremath{{\text{\rm Argmin}\,}}}

\newcommand{\diag}{\ensuremath{\text{\rm diag}\,}}
\newcommand{\dft}{\ensuremath{\text{\rm DFT}\,}}
\newcommand{\idft}{\ensuremath{\text{\rm IDFT}\,}}
\newcommand{\HH}{\ensuremath{{\mathcal H}}}

\newcommand{\GG}{\ensuremath{{\mathcal G}}}

\newcommand{\Sum}{\ensuremath{\displaystyle\sum}}

\newcommand{\emp}{\ensuremath{{\varnothing}}}
\newcommand{\Id}{\ensuremath{\text{\rm Id}}}

\newcommand{\RR}{\ensuremath{\mathbb{R}}}

\newcommand{\RPP}{\ensuremath{\left]0,+\infty\right[}}

\newcommand{\soft}[1]{\ensuremath{{\:\text{\rm soft}}_{{#1}}\,}}
\newcommand{\hard}[1]{\ensuremath{{\:\text{\rm hard}}_{{#1}}\,}}

\newcommand{\RX}{\ensuremath{\left]-\infty,+\infty\right]}}

\newcommand{\NN}{\ensuremath{\mathbb N}}

\newcommand{\intdom}{\ensuremath{\text{int\,dom}\,}}

\newcommand{\exi}{\ensuremath{\exists\,}}

\newcommand{\ran}{\ensuremath{\text{\rm ran}\,}}

\newcommand{\cran}{\ensuremath{\overline{\text{\rm ran}}\,}}
\newcommand{\zer}{\ensuremath{\text{\rm zer}\,}}
\newcommand{\pinf}{\ensuremath{{+\infty}}}

\newcommand{\dom}{\ensuremath{\text{\rm dom}\,}}

\newcommand{\prox}{\ensuremath{\text{\rm prox}}}

\newcommand{\proj}{\ensuremath{\text{\rm proj}}}

\newcommand{\Fix}{\ensuremath{\text{\rm Fix}\,}}

\newcommand{\card}{\ensuremath{\text{\rm card}\,}}
\newcommand{\sign}{\ensuremath{\,\text{\rm sign}}}
\newcommand{\gra}{\ensuremath{\text{\rm gra}\,}}
\newcommand{\inte}{\ensuremath{\text{\rm int}\,}}

\newcommand{\rzeroun}{\ensuremath{\left]0,1\right]}}

\def\abstract{\noindent{\bfseries Abstract}. \ignorespaces}

\newtheorem{theorem}{Theorem}[section]
\newtheorem{lemma}[theorem]{Lemma}

\newtheorem{proposition}[theorem]{Proposition}

\theoremstyle{plain}{\theorembodyfont{\rmfamily}%
\newtheorem{example}[theorem]{Example}}
\theoremstyle{plain}{\theorembodyfont{\rmfamily}%
\newtheorem{remark}[theorem]{Remark}}
\theoremstyle{plain}{\theorembodyfont{\rmfamily}%
}
\theoremstyle{plain}{\theorembodyfont{\rmfamily}%
}
\theoremstyle{plain}{\theorembodyfont{\rmfamily}%
}
\theoremstyle{plain}{\theorembodyfont{\rmfamily}%
\newtheorem{definition}[theorem]{Definition}}
\theoremstyle{plain}{\theorembodyfont{\rmfamily}%
\newtheorem{problem}[theorem]{Problem}}

\numberwithin{equation}{section}
\begin{document}

\title{\sffamily\huge A Variational Inequality Model for the
Construction of Signals from Inconsistent Nonlinear 
Equations\thanks{Contact 
author: P. L. Combettes, {\email{plc@math.ncsu.edu}},
phone:+1 (919) 515 2671. The work of P. L. Combettes was 
supported by the National Science Foundation under grant 
CCF-1715671 and the work of Z. C. Woodstock was supported by 
the National Science Foundation under grant DGE-1746939.}}

\author{Patrick L. Combettes and Zev C. Woodstock\\
\small North Carolina State University,
Department of Mathematics,
Raleigh, NC 27695-8205, USA\\
\small \email{plc@math.ncsu.edu}\,, \email{zwoodst@ncsu.edu}\\
}

\date{~}
\maketitle

\bigskip

\begin{abstract} 
Building up on classical linear formulations, we posit that a broad
class of problems in signal synthesis and in signal recovery are
reducible to the basic task of finding a point in a closed convex
subset of a Hilbert space that satisfies a number of nonlinear
equations involving firmly nonexpansive operators. We investigate
this formalism in the case when, due to inaccurate modeling or
perturbations, the nonlinear equations are inconsistent. A relaxed
formulation of the original problem is proposed in the form of a
variational inequality. The properties of the relaxed problem are
investigated and a provenly convergent block-iterative algorithm,
whereby only blocks of the underlying firmly nonexpansive operators
are activated at a given iteration, is devised to solve it.
Numerical experiments illustrate robust recoveries in several
signal and image processing applications.
\end{abstract}

\newpage
\section{Introduction}
\label{sec:1}

Signal construction encompasses forward problems such as image
synthesis, holography, filter design, time-frequency distribution
synthesis, and radiation therapy planning, as well as inverse
problems such as density estimation, signal denoising, image
interpolation, signal extrapolation, audio declipping, image 
reconstruction, or deconvolution; see, e.g., 
\cite{Avil17,Cens05,Aiep96,Ehrg10,Fouc18,Gold85,%
Renc19,Sale87,Samw18,Shak08,Whit92}. Essential components in the
mathematical modeling of signal construction problems are equations
tying the ideal solution $\overline{x}$ in a space $\HH$ to
given prescriptions in a space $\GG$, say $W\overline{x}=p$,
where $W$ is an operator mapping $\HH$ to $\GG$. 
The prescription $p$ can be a design specification in
forward problems, or an observation in inverse problems. 

In 1978, Youla \cite{Youl78} elegantly brought to light the simple
geometry that underlies many classical problems in signal
construction by reducing them to the following formulation: given
closed vector subspaces $C$ and $D$ in a real Hilbert space $\HH$,
and a point $p\in D$, 
\begin{equation}
\label{e:Youla78}
\text{find}\;\;x\in C\:\;\text{such that}\:\;\proj_{D}x=p,
\end{equation}
where $\proj_D$ denotes the projection operator onto $D$. In the
context of signal recovery, the original signal of interest
$\overline{x}$ is known to lie in $C$ and some observation $p$ of
it is available in the form of its projection onto $D$. A natural 
nonlinear extension of this setting is obtained by considering
nonempty closed convex sets $C$ in $\HH$ and $D$ in a real Hilbert
space $\GG$, a bounded linear operator $L\colon\HH\to\GG$, a point
$p\in D$, and setting as an objective to
\begin{equation}
\label{e:y1}
\text{find}\;\;x\in C\:\;\text{such that}\:\;\proj_{D}(Lx)=p.
\end{equation}
An early instance of this model appears in \cite{Abel91}, where
$C$ is a set of bandlimited signals and $p$ is an observation of
$N$ clipped samples of the original signal. Thus,
$L\colon\HH\to\RR^N$ is the sampling operator and 
$D=\menge{y\in\RR^N}{\|y\|_\infty\leq\rho}$ for some 
$\rho\in\RPP$. A key property of projectors onto closed convex 
sets is their firm nonexpansiveness. Recall that an operator 
$F\colon\GG\to\GG$ is firmly nonexpansive if \cite{Livre1}
\begin{equation}
\label{e:f}
(\forall x\in\GG)(\forall y\in\GG)\quad
\scal{x-y}{Fx-Fy}\geq\|Fx-Fy\|^2.
\end{equation}
In \cite{Eusi20,Ibap21}, it was shown that many nonlinear
observation processes found in signal processing, machine learning,
and inference problems can be represented through such operators.
This prompts us to consider the following formulation, whereby the
prescriptions are modeled via Wiener systems (see 
Figure~\ref{fig:w}).

\begin{figure}[h!tb]
\begin{center}
\begin{circuitikz} 
\draw (0,1.5) node[left,label={[xshift=-0.75mm,
    yshift=-3.25mm]$\overline{x}$}](x){};
\draw (4.0,3.0) node[right](p1){$p_1$};
\draw (4.0,0) node[right](pm){$p_m$};
\draw (4.0,1.5) node[right](pi){$p_i$};
\draw (2.40,0.85) node[] {$\vdots$};
\draw (2.40,2.35) node[] {$\vdots$};
\draw[line width=0.28mm]
(x) to[] (0.5,1.5)
    to[] (0.5,3.0)
    to[] (1.75,3.0) node[draw,fill=white,minimum
    size=8pt, inner sep=4.5pt](l1){$L_{1}$}
    to[] (3.0,3.0) node[draw,fill=white,minimum
    size=8pt, inner sep=4.5pt](f1){$F_{1}$};
\draw[-latex,line width=0.25mm] (f1.east) -- (p1.west);
\draw[-latex,line width=0.25mm] (1.25,3.0) -- (l1.west);
\draw[-latex,line width=0.25mm] (l1.east) -- (f1.west);
\draw[line width=0.28mm]
(x) to[] (0.5,1.5)
    to[] (0.5,0)
    to[] (1.75,0) node[draw,fill=white,minimum
    size=8pt, inner sep=4pt](lm){$L_{m}$}
    to[] (3.0,0) node[draw,fill=white,minimum
    size=8pt, inner sep=4pt](fm){$F_{m}$};
\draw[-latex,line width=0.25mm] (fm.east) -- (pm.west);
\draw[-latex,line width=0.25mm] (1.25,0) -- (lm.west);
\draw[-latex,line width=0.25mm] (lm.east) -- (fm.west);
\draw[line width=0.28mm]
(x) to[] (0.5,1.5)
    to[] (0.5,1.5)
    to[] (1.75,1.5) node[draw,fill=white,minimum
    size=8pt, inner sep=5.2pt](li){$L_{i}$}
    to[] (3.0,1.5) node[draw,fill=white,minimum
    size=8pt, inner sep=5.2pt](fi){$F_{i}$};
\draw[-latex,line width=0.25mm] (fi.east) -- (pi.west);
\draw[-latex,line width=0.25mm] (1.25,1.5) -- (li.west);
\draw[-latex,line width=0.25mm] (li.east) -- (fi.west);
\end{circuitikz}
\end{center}
\caption{Illustration of Problem~\ref{prob:1} with $m$
prescriptions $(p_i)_{1\leq i\leq m}$. The $i$th prescription
$p_i$ is the output produced when the ideal signal $\overline{x}$
is input to a Wiener system $W_i=F_i\circ L_i$, i.e., the
concatenation of a linear system $L_i$ and a nonlinear system $F_i$
\cite{Sche81}. In the proposed model, $F_i$ is a firmly
nonexpansive operator.}
\label{fig:w}
\end{figure}
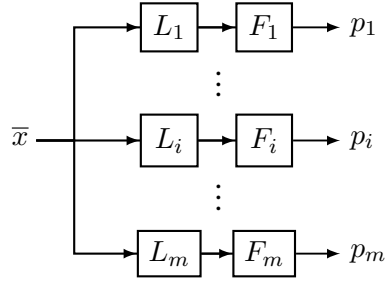

\begin{problem}
\label{prob:1}
Let $I$ be a nonempty finite set and let $C$ be a nonempty closed
convex subset of a real Hilbert space $\HH$. For every $i\in I$, 
let $\GG_i$ be a real Hilbert space, let $p_i\in\GG_i$, let
$L_i\colon\HH\to\GG_i$ be a nonzero bounded linear operator, and
let $F_i\colon\GG_i\to\GG_i$ be a firmly nonexpansive operator. The
task is to 
\begin{equation}
\label{e:prob1}
\text{find}\:\;x\in C\:\;\text{such that}\:\;
(\forall i\in I)\:\;F_i(L_ix)=p_i.
\end{equation}
\end{problem}

The work of \cite{Eusi20,Ibap21} assumes that the prescription 
equations in Problem~\ref{prob:1} are exact and hence that a
solution exists. In many instances, however, the prescription
operators may be imperfectly known or the model may be corrupted
by perturbations, so that Problem~\ref{prob:1} may not have 
solutions, e.g., \cite{Cens18,Sign94,Fouc18}. 
A dramatic consequence of this lack of feasibility is
that the algorithms proposed \cite{Eusi20,Ibap21} are known to
diverge in such situations. To deal robustly with possibly
inconsistent equations, one must therefore come up with an
appropriate relaxed formulation of Problem~\ref{prob:1}, i.e.,
one that seeks a point in $C$ that satisfies the nonlinear
equations in an approximate sense, and coincides with
the original problem \eqref{e:prob1} if it happens to be
consistent. To guide our design of a relaxed problem, let us
consider a classical instantiation of Problem~\ref{prob:1}.

\begin{example}
\label{ex:1}
Specialize Problem~\ref{prob:1} by setting, for every $i\in I$,
\begin{equation}
\label{e:10}
p_i=0\;\;\text{and}\;\;F_i=\Id-\proj_{D_i},\;\;\text{where $D_i$
is a nonempty closed convex subset of $\GG_i$},
\end{equation}
and note that the operators $(F_i)_{i\in I}$ are firmly
nonexpansive \cite[Corollary~4.18]{Livre1}. In this context,
\eqref{e:prob1} reduces to the convex feasibility problem 
\cite{Cens94,Aiep96,Youl82} 
\begin{equation}
\label{e:prob3}
\text{find}\:\;x\in C\:\;\text{such that}\:\;
(\forall i\in I)\:\;L_ix\in D_i.
\end{equation}
Let $(\omega_i)_{i\in I}$ be real numbers in $\rzeroun$ such that
$\sum_{i\in I}\omega_i=1$ and, for every $i\in I$, let $d_{D_i}$ be
the distance function to $D_i$. As seen in \cite{Siim19} (see also
\cite{Cens05,Cens18,Sign94,Sign99,Gold85,Youl86} for special
cases), a relaxation of \eqref{e:prob3} when it may be inconsistent
is the least-squares problem 
\begin{equation}
\label{e:prob4}
\minimize{x\in C}{f(x)},\quad\text{where}\quad 
f\colon x\mapsto\frac{1}{2}\sum_{i\in I}\,\omega_id_{D_i}^2(L_ix)
=\frac{1}{2}\sum_{i\in I}\,\omega_i\|L_ix-\proj_{D_i}(L_ix)\|^2.
\end{equation}
An important property of this formulation is that $f$ is a smooth
convex function since \cite[Corollary~12.31]{Livre1} asserts that
\begin{equation}
\label{e:fm1}
(\forall i\in I)\quad
\nabla\dfrac{d_{D_i}^2\circ L_i}{2}=
L_i^*\circ(\Id-\proj_{D_i})\circ L_i=
L_i^*\circ F_i\circ L_i-L_i^*p_i.
\end{equation}
It can therefore be solved by the
projection-gradient algorithm \cite[Corollary~28.10]{Livre1}.
Let us also note that \eqref{e:prob4} is a valid relaxation of
\eqref{e:prob3}. Indeed, if the latter has solutions, then $f$
vanishes on $C$ at those points only, and \eqref{e:prob4} is
therefore equivalent to \eqref{e:prob3}.
Historically, the first instance of the above relaxation process
seems to be Legendre's least-squares methods \cite{Lege05}. There,
$\HH=\RR^N=C$ and, for every $i\in I$, $\GG_i=\RR$,
$D_i=\{\beta_i\}$, and $L_i=\scal{\cdot}{a_i}$, where 
$\beta_i\in\RR$ and $0\neq a_i\in\RR^N$. 
Set $b=(\beta_i)_{i\in I}$, let $A$ be the matrix with 
rows $(a_i)_{i\in I}$, and let
$(\forall i\in I)$ $\omega_i=1/\card I$. Then \eqref{e:prob3}
consists of solving the linear system $Ax=b$ and \eqref{e:prob4} of
minimizing the function $x\mapsto\|Ax-b\|^2$.
\end{example}

In general, there is no suitable relaxation of Problem~\ref{prob:1}
in the form of a tractable convex minimization problem such as
\eqref{e:prob4}. For instance, in Example~\ref{ex:1}, we can 
rewrite \eqref{e:prob4} as
\begin{equation}
\label{e:prob14}
\minimize{x\in C}{f(x)},\quad\text{where}\quad 
f\colon x\mapsto\frac{1}{2}\sum_{i\in I}\,\omega_i
\|F_i(L_ix)-p_i\|^2.
\end{equation}
However, beyond the special case \eqref{e:10}, $f$ is typically a
nonconvex and nondifferentiable function
\cite{Avil17,Pete19,Zarz09}, which makes it impossible to 
guarantee the construction of solutions. Another plausible
formulation that captures \eqref{e:prob4} would be to introduce in
Problem~\ref{prob:1} the closed convex sets
$(\forall i\in I)$ $D_i=\menge{y\in\GG_i}{F_iy_i=p_i}$.
However the resulting minimization problem \eqref{e:prob4} is
intractable because we typically do not know how to evaluate the
operators $(\proj_{D_i})_{i\in I}$, and therefore cannot evaluate
$f$ and its gradient.

Our strategy to relax \eqref{e:prob1} is to forego the optimization
approach in favor of the broader framework of \emph{variational
inequalities}. To motivate this approach, let us go back to
Example~\ref{ex:1}. Then it follows from Lemma~\ref{l:1} below and
\eqref{e:fm1} that \eqref{e:prob4} equivalent to finding $x\in C$
such that $(\forall y\in C)$ 
$\sum_{i\in I}\omega_i \scal{L_i(y-x)}{F_i(L_ix)-p_i}\geq 0$. 
We shall show that this variational inequality constitutes an
appropriate relaxed formulation of Problem~\ref{prob:1} in the
presence of general firmly nonexpansive operators $(F_i)_{i\in I}$,
and that it can be solved iteratively through an efficient
block-iterative fixed point algorithm. Here is a precise
formulation of our relaxed problem.

\begin{problem}
\label{prob:2}
Let $I$ be a nonempty finite set, let $(\omega_i)_{i\in I}$ be real
numbers in $\rzeroun$ such that $\sum_{i\in I}\omega_i=1$, and let
$C$ be a nonempty closed convex subset of a real Hilbert space
$\HH$. For every $i\in I$, let $\GG_i$ be a real Hilbert space, let
$p_i\in\GG_i$, let $L_i\colon\HH\to\GG_i$ be a nonzero bounded
linear operator, and let $F_i\colon\GG_i\to\GG_i$ be a firmly
nonexpansive operator. The task is to 
\begin{equation}
\label{e:prob2}
\text{find}\:\;x\in C\:\;\text{such that}\:\;
(\forall y\in C)\;\:\sum_{i\in I}\omega_i
\scal{L_i(y-x)}{F_i(L_ix)-p_i}\geq 0.
\end{equation}
\end{problem}

The paper is organized as follows. Section~\ref{sec:2} provides the
notation and the necessary background, as well as preliminary
results. It covers in particular the basics of monotone operator
theory, which will play an essential role in the paper. In
Section~\ref{sec:3}, we illustrate the flexibility and the breadth
the proposed firmly nonexpansive Wiener model. In
Section~\ref{sec:4}, we analyze various properties of
Problem~\ref{prob:2}, in particular as a relaxation of
Problem~\ref{prob:1}. We also provide in that section a
block-iterative algorithm to solve Problem~\ref{prob:2}.
Section~\ref{sec:5} is devoted to numerical experiments in the area
of signal and image processing.

\section{Notation, background, and preliminary results}
\label{sec:2}

\subsection{Notation} 
\label{sec:21}
Our notation follows \cite{Livre1}, to which
one can refer for background on monotone operators and convex
analysis. Let $\HH$ be a real Hilbert space with scalar product
$\scal{\cdot}{\cdot}$, associated norm $\|\cdot\|$, and identity
operator $\Id$. The family of all subsets of $\HH$ is denoted by
$2^{\HH}$. The Hilbert direct sum of a family of real
Hilbert spaces $(\HH_i)_{i\in I}$ is denoted by
$\bigoplus_{i\in I}\HH_i$.

Let $T\colon\HH\to\HH$. Then $T$ is \emph{cocoercive} if there
exists $\beta\in\RPP$ such that 
\begin{equation}
\label{e:coco}
(\forall x\in\HH)(\forall y\in\HH)\quad
\scal{x-y}{Tx-Ty}\geq\beta\|Tx-Ty\|^2,
\end{equation}
and \emph{firmly nonexpansive} if $\beta=1$ above.
The set of \emph{fixed points} of $T$ is 
$\Fix T=\menge{x\in\HH}{Tx=x}$. 

Let $A\colon\HH\to 2^{\HH}$. The \emph{graph} of $A$ is 
$\gra A=\menge{(x,x^*)\in\HH\times\HH}{x^*\in Ax}$, the
\emph{domain} of
$A$ is $\dom A=\menge{x\in\HH}{Ax\neq\emp}$,
the \emph{range} of
$A$ is $\ran A=\menge{x^*\in\HH}{(\exi x\in\HH)\;x^*\in Ax}$,
the set of zeros of $A$ is $\zer A=\menge{x\in\HH}{0\in Ax}$, 
the \emph{inverse} of $A$ is 
$A^{-1}\colon\HH\to 2^{\HH}\colon x^*\mapsto
\menge{x\in\HH}{x^*\in Ax}$, and the \emph{resolvent} of $A$ is
$J_A=(\Id+A)^{-1}$. Further, $A$ is \emph{monotone} if
\begin{equation}
\big(\forall(x,x^*)\in\gra A\big)
\big(\forall(y,y^*)\in\gra A\big)\quad
\scal{x-y}{x^*-y^*}\geq 0,
\end{equation}
and \emph{maximally monotone} if, for every 
$(x,x^*)\in\HH\times\HH$, 
\begin{equation} 
\label{e:maxmon2}
(x,x^*)\in\gra A\quad\Leftrightarrow\quad
\big(\forall (y,y^*)\in\gra A\big)\;\;\scal{x-y}{x^*-y^*}\geq 0.
\end{equation}
If $A$ is maximally monotone, then $J_A$ is a single-valued firmly
nonexpansive operator defined on $\HH$. If $A$ is monotone and
satisfies
\begin{equation}
\label{e:3*mon}
(\forall (x,x^*)\in\dom A\times\ran A)\;\;\sup
\menge{\scal{x-y}{y^*-x^*}}{(y,y^*)\in\gra A}<\pinf,
\end{equation}
then it is \emph{$3^*$ monotone}.

$\Gamma_0(\HH)$ is the class of all lower semicontinuous convex
functions from $\HH$ to $\RX$ which are proper in the sense that
they are not identically $\pinf$. Let $f\in\Gamma_0(\HH)$. The
\emph{domain} of $f$ is $\dom f=\menge{x\in\HH}{f(x)<\pinf}$, the
\emph{conjugate} of $f$ is the function 
\begin{equation}
\label{e:conj}
\Gamma_0(\HH)\ni f^*\colon x^*\mapsto
\sup_{x\in\HH}\big(\scal{x}{x^*}-f(x)\big),
\end{equation}
and the \emph{subdifferential} of $f$ is the maximally monotone
operator
\begin{equation}
\label{e:subdiff}
\partial f\colon\HH\to 2^{\HH}\colon
x\mapsto\menge{x^*\in\HH}{(\forall y\in\HH)\;
\scal{y-x}{x^*}+f(x)\leq f(y)}. 
\end{equation}
The \emph{Moreau envelope} of $f$ is 
\begin{equation}
\label{e:18}
\widetilde{f}\colon\HH\to\RR\colon x\mapsto\inf_{y\in\HH}
\bigg(f(y)+\dfrac{\|x-y\|^2}{2}\bigg).
\end{equation}
For every $x\in\HH$, the infimum in \eqref{e:18} is achieved at a
unique point, which is denoted by $\prox_fx$. This defines
the \emph{proximity operator} $\prox_f=J_{\partial f}$ of $f$.

Let $C$ be a nonempty closed and convex subset of $\HH$. The
\emph{distance} from
$x\in\HH$ to $C$ is $d_C(x)=\inf_{y\in C}\|x-y\|$, the
\emph{indicator function} of $C$ is 
\begin{equation}
\label{e:iota}
\iota_C\colon\HH\to\RX\colon x\mapsto 
\begin{cases}
0,&\text{if}\;\:x\in C;\\
\pinf,&\text{if}\;\:x\notin C,
\end{cases}
\end{equation}
the \emph{normal cone} to
$C$ at $x\in\HH$ is 
\begin{equation} 
\label{e:normalcone}
N_Cx=\partial\iota_C(x)=
\begin{cases}
\menge{x^*\in\HH}{(\forall y\in C)\;\;\scal{y-x}{x^*}\leq 0},
&\text{if}\;\;x\in C;\\
\emp,&\text{otherwise,}
\end{cases}
\end{equation}
and the \emph{projection operator} onto $C$ is 
$\proj_C=\prox_{\iota_C}=J_{N_C}$. 

The following facts will also come into play.

\begin{lemma}
\label{l:34}
Let $A\colon\HH\to2^\HH$ be maximally monotone, let $\mu\in\RPP$,
and let $\gamma\in\left]0,1/\mu\right[$. Set $B=A-\mu\Id$ and
$\beta=1-\gamma\mu$. Then $J_{\gamma B}\colon\HH\to\HH$ is
$\beta$-cocoercive. Furthermore,
$J_{\gamma B}=J_{\beta^{-1}\gamma A}\circ(\beta^{-1}\Id)$.
\end{lemma}
\begin{proof}
Let $x$ and $q$ be in $\HH$. Since $\beta^{-1}\gamma A$ is 
maximally monotone, its resolvent is single-valued with domain
$\HH$. Therefore, 
\begin{align}
q\in J_{\gamma B}x&\Leftrightarrow x-q\in\gamma Bq
\nonumber\\
&\Leftrightarrow x-\beta q\in\gamma Aq
\nonumber\\
&\Leftrightarrow\beta^{-1}x-q
\in\beta^{-1}\gamma Aq\nonumber\\
&\Leftrightarrow
q=J_{\beta^{-1}\gamma A}\big(\beta^{-1}x\big),
\end{align}
which shows that 
$J_{\gamma B}=J_{\beta^{-1}\gamma A}\circ(\beta^{-1}\Id)$
is single-valued with domain $\HH$.
Finally, since $M=\beta\gamma A$ is maximally monotone, it follows
from \cite[Corollary~23.26]{Livre1} that 
$J_{\gamma B}=J_{\beta^{-2}M}\circ(\beta^{-1}\Id)$ is
$\beta$-cocoercive.
\end{proof}

\begin{lemma}%
[\hspace{1sp}{{\cite[Proposition~24.68]{Livre1}}}]
\label{l:m1}
Let $\HH$ be the real Hilbert space of $N\times M$ matrices
under the Frobenius norm, and set $s=\min\{N,M\}$. Denote the 
singular value decomposition of $x\in\HH$ by 
$x=U_x\:\diag(\sigma_1(x),\ldots,\sigma_s(x))V_x^\top$. Let
$\phi\in\Gamma_0(\RR)$ be even, and set
\begin{equation}
\label{e:Fm}
F\colon\HH\to\HH\colon x\mapsto U_x\:\diag
\Big(\prox_{\phi}\big(\sigma_1(x)\big),\ldots,\prox_{\phi}
\big(\sigma_s(x)\big)\Big)V_x^\top.
\end{equation}
Then $F$ is firmly nonexpansive.
\end{lemma}

\subsection{Variational inequalities}
\label{sec:22}

The following notion of a \emph{variational inequality} was
formulated in \cite{Brow65} (see Figure~\ref{fig:0}). 

\begin{definition}
\label{d:vi}
Let $C$ be a nonempty closed convex set of $\HH$ and let
$B\colon\HH\to\HH$ be a monotone operator. The associated
\emph{variational inequality} problem is to 
\begin{equation}
\label{e:0}
\text{find}\;\;x\in C\:\;\text{such that}\:\;
(\forall y\in C)\:\;\scal{y-x}{Bx}\geq 0.
\end{equation}
\end{definition}

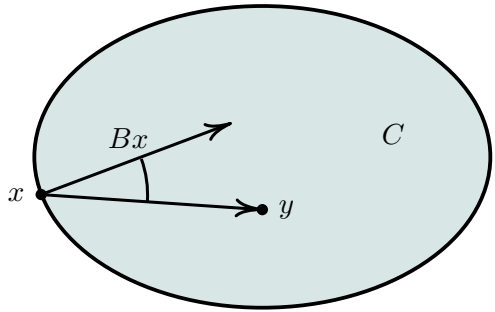
\begin{figure}[h!tb]
\begin{center}
\begin{tikzpicture}
\definecolor{color2b}{rgb}{0.85,0.90,0.90}
\draw[-latex,black,fill=color2b, line width=0.5mm] (0, 0) ellipse
[x radius = 3cm, y radius=2cm, start angle=30, end angle=150];
\draw[] (-2.91,-0.5) node[circle,fill,inner
sep=1.5pt,label=left:$x$](a){}; 
\draw [] (0,-0.7) node[circle,fill,inner
sep=1.5pt,label=right:$y$](b){};
\draw [] (1.3,0.3) node[label=right:$C$](c){};
\draw [] (-2.3,0.25) node[label=right:$Bx$](d){};
\draw [] (-0.4,0.4) node[](e){};
\draw[-{>[scale=1.7,length=6,width=4.2]},line width=0.4mm] 
(-2.91,-0.5 ) to (0,-0.7);
\draw[-{>[scale=1.7,length=6,width=4.2]},line width=0.4mm] 
(-2.91,-0.5 ) to (-0.4,0.45);
\pic [draw,angle radius = 1.4cm,line width=0.4mm,angle 
eccentricity=1.3]{angle = b--a--e};
\end{tikzpicture}
\end{center}
\caption{Illustration of the variational inequality principle.
The point $x$ solves \eqref{e:0} because it lies in $C$ and, for
every $y\in C$, the angle between $y-x$ and $Bx$ is acute.}
\label{fig:0}
\end{figure}

Variational inequalities are used in various areas of mathematics 
and its applications \cite{Bens78,Facc03,Kind80,Zeid90}. They are
also central in constrained minimization problems.

\begin{lemma}{\rm\cite[Proposition~27.8]{Livre1}}
\label{l:1}
Let $f\colon\HH\to\RR$ be a differentiable convex function, 
let $C$ be a nonempty closed convex subset of $\HH$, and let 
$x\in\HH$. Then $x$ minimizes $f$ over $C$ if and only if
it satisfies the variational inequality
\begin{equation}
\label{e:15}
x\in C\:\;\text{and}\:\;
(\forall y\in C)\:\;\scal{y-x}{\nabla f(x)}\geq 0.
\end{equation}
\end{lemma}

\subsection{Composite sums of monotone operators}

We shall require the following Br\'ezis--Haraux-type theorem, which
remains valid in general reflexive Banach spaces (see
\cite[Th\'eor\`eme~3]{Breh76} for the special case of the sum 
of two monotone operators).

\begin{lemma}
\label{l:2}
Let $\HH$ be a real Hilbert space and let 
$(\GG_i)_{i\in I}$ be a finite family of real Hilbert
spaces. Let $A\colon\HH\to 2^{\HH}$ be a $3^*$ monotone operator
and, for every $i\in I$, let $B_i\colon\GG_i\to 2^{\GG_i}$ be a
$3^*$ monotone operator and let $L_i\colon\HH\to\GG_i$ be a bounded
linear operator. Suppose that 
$A+\sum_{i\in I}L_i^*\circ B_i\circ L_i$ is maximally monotone.
Then 
\begin{equation}
\label{e:p24}
\begin{cases}
\inte\big(\ran A+\sum_{i\in I}L_i^*(\ran B_i)\big)=
\inte\ran\big(A+\sum_{i\in I}L_i^*\circ B_i\circ L_i\big)\\[3mm]
\overline{\ran A+\sum_{i\in I}L_i^*(\ran B_i)}=
\cran\big(A+\sum_{i\in I}L_i^*\circ B_i\circ L_i\big).
\end{cases}
\end{equation}
\end{lemma}
\begin{proof}
Clearly, $\ran(A+\sum_{i\in I}L_i^*\circ B_i\circ L_i)
\subset(\ran A+\sum_{i\in I}L_i^*(\ran B_i))$. It is therefore
enough to show that 
\begin{equation}
\label{e:p25}
\begin{cases}
\inte\big(\ran A+\sum_{i\in I}L_i^*(\ran B_i)\big)\subset
\ran\big(A+\sum_{i\in I}L_i^*\circ B_i\circ L_i\big)\\
\ran A+\sum_{i\in I}L_i^*(\ran B_i)\subset
\cran\big(A+\sum_{i\in I}L_i^*\circ B_i\circ L_i\big).
\end{cases}
\end{equation}
Without loss of generality, set $I=\{1,\ldots,m\}$ and introduce
the Hilbert direct sum
$\HHH=\HH\oplus\GG_1\oplus\cdots\oplus\GG_m$. Furthermore, 
introduce the bounded linear operator
$\boldsymbol{L}\colon\HH\to\HHH\colon x\mapsto(x,L_1x,\ldots,L_mx)$
and the operator 
$\boldsymbol{M}\colon\HHH\to 2^{\HHH}\colon (x,y_1,\ldots,y_m)
\mapsto Ax\times B_1y_1\times\cdots\times B_my_m$, which is
$3^*$ monotone since $A$ and $(B_i)_{i\in I}$ are. Note also that,
since $\boldsymbol{L}^*\colon\HHH\to\HH\colon
(x,y_1,\ldots,y_m)\mapsto x+\sum_{i\in I}L_i^*y_i$,
the operator
\begin{equation}
\boldsymbol{L}^*\circ \boldsymbol{M}\circ \boldsymbol{L}=
A+\sum_{i\in I}L_i^*\circ B_i\circ L_i 
\end{equation}
is maximally monotone. 
We can therefore apply \cite[Theorem~5]{Penn01} to obtain 
\begin{equation}
\begin{cases}
\inte\boldsymbol{L}^*(\ran\boldsymbol{M})\subset
\ran(\boldsymbol{L}^*\circ\boldsymbol{M}\circ\boldsymbol{L})\\
\boldsymbol{L}^*(\ran\boldsymbol{M})\subset
\cran(\boldsymbol{L}^*\circ\boldsymbol{M}\circ\boldsymbol{L}),
\end{cases}
\end{equation}
which is precisely \eqref{e:p25}.
\end{proof}

We consider below a monotone inclusion problem involving several
operators.

\begin{problem}
\label{prob:20}
Let $(\omega_i)_{i\in I}$ be a finite family of real numbers in
$\rzeroun$ such that $\sum_{i\in I}\omega_i=1$, let
$A_0\colon\HH\to 2^{\HH}$ be maximally monotone and, for every
$i\in I$, let $\beta_i\in\RPP$ and let $A_i\colon\HH\to\HH$ be
$\beta_i$-cocoercive. The task is to find $x\in\HH$ such that $0\in
A_0x+\sum_{i\in I}\omega_i A_ix$. 
\end{problem}

\begin{proposition}{\rm\cite[Proposition~4.9]{Anon21}}
\label{p:20}
Consider the setting of Problem~\ref{prob:20} under
the assumption that it has a solution. Let $K$ be a strictly 
positive integer and let $(I_n)_{n\in\NN}$ be a sequence of 
nonempty subsets of $I$ such that $(\forall n\in\NN)$ 
$\bigcup_{k=0}^{K-1}I_{n+k}=I$.
Let $\gamma\in\left]0,2\min_{1\leq i\leq m}\beta_i\right[$,
let $x_0\in\HH$, and let $(\forall i\in I)$ $t_{i,-1}\in\HH$.
Iterate
\begin{equation}
\label{e:a30}
\begin{array}{l}
\text{for}\;n=0,1,\ldots\\
\left\lfloor
\begin{array}{l}
\text{for every}\;i\in I_n\\
\left\lfloor
\begin{array}{l}
t_{i,n}=x_n-\gamma A_ix_n\\
\end{array}
\right.\\
\text{for every}\;i\in I\smallsetminus I_n\\
\left\lfloor
\begin{array}{l}
t_{i,n}=t_{i,n-1}\\
\end{array}
\right.\\[1mm]
x_{n+1}=J_{\gamma A_0}\Bigg(\Sum_{i\in I}\omega_it_{i,n}\Bigg).
\end{array}
\right.\\
\end{array}
\end{equation}
Then $(x_n)_{n\in\NN}$ converges weakly to a solution to
Problem~\ref{prob:20}.
\end{proposition}

\section{Firmly nonexpansive Wiener models}
\label{sec:3}

The proposed Wiener model (see Figure~\ref{fig:w}) involves a
linear operator followed by a firmly nonexpansive operator 
acting on a real Hilbert space $\GG$. Typical examples of linear
transformations in the context of signal construction include the
Fourier transform, the Radon transform, wavelet decompositions,
frame decompositions, audio effects, or blurring operators. We show
that firmly nonexpansive operators model many useful nonlinearities
in this context. Key examples based on those of \cite{Ibap21} are
recalled and new ones are proposed. Following \cite{Ibap21}, we
call $p\in\GG$ a \emph{proximal point} of $y\in\GG$ relative to a
firmly nonexpansive operator $F\colon\GG\to\GG$ if $Fy=p$. 

\subsection{Projection operators}
\label{sec:31}

As seen in Section~\ref{sec:21}, the projection operator onto a
nonempty closed convex set is firmly nonexpansive.

\begin{example}
\label{ex:hc}
For every $j\in\{1,\ldots,m\}$, let $\mathsf{G}_j$ be a real 
Hilbert space and let $\mathsf{D}_j\subset\mathsf{G}_j$ be nonempty
closed and convex. Suppose that 
$\GG=\bigoplus_{1\leq j\leq m}\mathsf{G}_j$.
The operator
\begin{equation}
\label{e:hc}
F\colon(\mathsf{y}_j)_{1\leq j\leq m}
\mapsto (\proj_{\mathsf{D}_j}\mathsf{y}_j)_{1\leq j\leq m},
\end{equation}
which is also the projection onto the closed convex set
$D=\bigtimes_{1\leq j\leq m}\mathsf{D}_j$, is the hard clipper of
\cite[Example~2.11]{Ibap21}. If we specialize to the case when, for
every $j\in\{1,\ldots,m\}$, $\mathsf{G}_j=\RR$, we obtain the
standard hard clipping operators of \cite{Abel91,Fouc18,Tesh19}.
\end{example}

\begin{example}
\label{ex:k}
Let $K\subset\GG$ be a nonempty closed convex cone. The operator 
$F=\proj_{K}$ is used as a distortion model when $K$ is the 
positive orthant \cite[Section~10.4.1]{Tarr19}. Another instance
of a conic projection operator arises in isotonic regression 
\cite{Barl72}, where $K=\menge{(\xi_i)_{1\leq i\leq N}\in\RR^N}
{\xi_1\leq\cdots\leq\xi_N}$.
\end{example}

\begin{example}
Compression schemes such as downsampling project a high-dimensional
object of interest onto a closed convex subset of a low-dimensional
subspace of $\GG$ \cite{Nasr14}.
\end{example}

\subsection{Proximity operators}
\label{sec:32}

As seen in Section~\ref{sec:21}, the proximity operator of a
function in $\Gamma_0(\GG)$ is firmly nonexpansive. The following
construction will be particularly useful.

\begin{example}
\label{ex:80}
For every $j\in\{1,\ldots,m\}$, let $\mathsf{G}_j$ be a real 
Hilbert space and let $\mathsf{g}_j\in\Gamma_0(\mathsf{G}_j)$. 
Suppose that $\GG=\bigoplus_{1\leq j\leq m}\mathsf{G}_j$ and set
$F\colon\GG\to\GG\colon 
(\mathsf{y}_j)_{1\leq j\leq m}\mapsto
(\prox_{\mathsf{g}_i}\mathsf{y}_j)_{1\leq j\leq m}$.
Then \cite[Proposition~24.11]{Livre1} asserts that
\begin{equation}
\label{e:79}
F=\prox_g,\quad\text{where}\quad
g\colon\GG\to\left]-\infty,+\infty\right]\colon
(\mathsf{y}_j)_{1\leq j\leq m}\mapsto\sum_{j=1}^m 
\mathsf{g}_j(\mathsf{y}_j).
\end{equation}
\end{example}

\begin{example}
\label{ex:81}
In Example~\ref{ex:80} suppose that, for every 
$j\in\{1,\ldots,m\}$, $\mathsf{g}_j=\phi_j\circ\|\cdot\|$, where
$\phi_j$ is an even function in $\Gamma_0(\RR)$ such that 
$\phi_j(0)=0$ and $\phi_j\neq\iota_{\{0\}}$.
Set $(\forall j\in\{1,\ldots,m\})$ 
$\rho_j=\text{\rm max}\,\partial\phi_j(0)$. Then we derive from
\cite[Proposition~2.1]{Nmtm09} that
\begin{multline}
\label{e:o1}
F\colon\GG\to\GG\colon (\mathsf{y}_j)_{1\leq j\leq m}\mapsto
\Big(\big(\prox_{\phi_j}\|\mathsf{y}_j\|\big)
\lfloor\mathsf{y}_j\rfloor_{\rho_j}\Big)_{1\leq j\leq m},\\
\text{where}\quad\lfloor\mathsf{y}_j\rfloor_{\rho_j}=
\begin{cases}
{\mathsf{y}}_j/\|{\mathsf{y}}_j\|,&\text{if}
\;\;\|\mathsf{y}_j\|>\rho_j;\\
\mathsf{0},&\text{if}\;\;\|\mathsf{y}_j\|\leq\rho_j.
\end{cases}
\end{multline}
\end{example}

\begin{example}
\label{ex:82}
Consider the special case of Example~\ref{ex:81} in which, for some
$j\in\{1,\ldots,m\}$, $\phi_j$ is not differentiable at the origin,
which implies that $\rho_j>0$. Then $\prox_{\mathsf{g}_j}$ acts as
a thresholder with respect to the $j$th variable in the sense that,
if $\|\mathsf{y}_j\|\leq\rho_j$, then the $j$th coordinate of $Fy$
is zero. For instance, suppose that, for every 
$j\in\{1,\ldots,m\}$, $\phi_j=\rho_j|\cdot|$, hence
$\partial\phi_j(0)=[-\rho_j,\rho_j]$ and 
$\mathsf{g}_j=\rho_j\|\cdot\|$. Then $Fy=p$ is acquired
though the group-shrinkage operation \cite{Yuan06}
\begin{equation}
p=\bigg(\bigg(1-\frac{\rho_j}
{\max\{\|\mathsf{y}_j\|,\rho_j\}}\bigg)
\mathsf{y}_j\bigg)_{1\leq j\leq m}.
\end{equation}
\end{example}

\begin{example}
\label{ex:83}
In contrast to the hard clipping operations of Example~\ref{ex:hc},
soft clipping operators are not projection operators in general,
but many turn out to be proximity operators \cite{Ibap21}
(see Figure~\ref{fig:sc}). For instance, consider the setting of
Example~\ref{ex:81} with 
\begin{equation}
(\forall j\in\{1,\ldots,m\})\quad
\phi_j\colon\eta\mapsto\begin{cases}
-|\eta|-\ln(1-|\eta|)-\dfrac{\eta^2}{2},&\text{if}\;\;|\eta|<1;\\
\pinf, &\text{if}\;\;|\eta|\geq 1.
\end{cases}
\end{equation}
Then we obtain the soft clipping operator
\begin{equation}
F\colon (\mathsf{y}_j)_{1\leq j\leq m}\mapsto 
\bigg(\dfrac{\mathsf{y}_j}{1+\|\mathsf{y}_j\|}
\bigg)_{1\leq j\leq m}
\end{equation}
used in \cite{Marm20}. Soft clipping operators model sensors in
signal processing \cite{Avil17,Marm20,Tarr19} and activation 
functions in neural networks \cite{Svva20}.
\end{example}

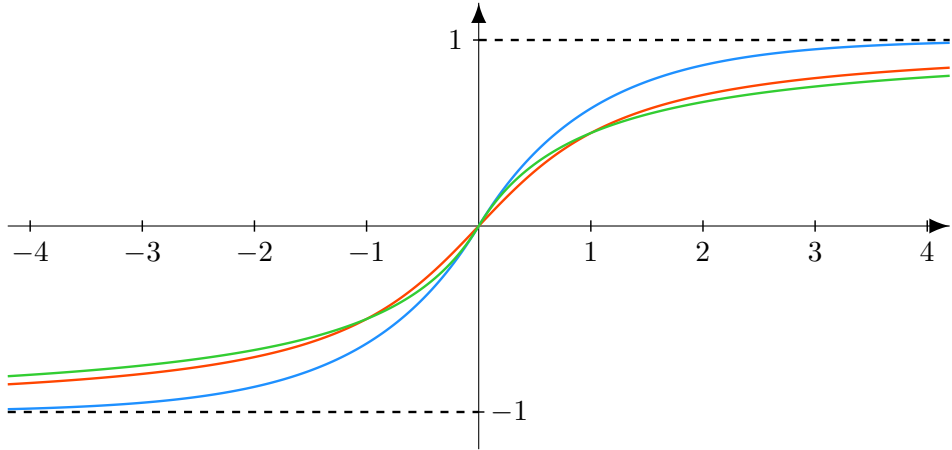
\begin{figure}[H]
\centering
\begin{tabular}{@{}c@{}c@{}}
\makeatletter
\tikzset{
nomorepostaction/.code=\makeatletter\let\tikz@postactions\pgfutil@empty, 
my axis/.style={
 postaction={
     decoration={
         markings,
         mark=at position 1 with {
             \arrow[ultra thick]{latex}
         }
     },
     decorate,
     nomorepostaction
 },
 thin, -, 
 every path/.append style=my axis 
    }
}
\makeatother
\begin{tikzpicture}
\begin{axis}[
  width=13.975cm,
  height=7.5cm,
  ymin=-1.2,ymax=1.2,
  axis lines=middle,
  axis line style={my axis},
  clip=false,
  ytick={1},
  yticklabels={$1$},
  ticklabel style={font=\normalsize},
  tick style={color=black, line width=0.05em},
  extra y ticks={-1},
  extra y tick style={y tick label style={right,
  xshift=0.25em},font=\small},
  extra y tick labels={$-1$},
]
\addplot [line width=0.9pt,mark=none,pblue,
samples=250,domain=-4.2:4.2] {sign(x)*(1-exp(-abs(x)))};
\addplot [line width=0.9pt,mark=none,pred,
samples=250,domain=-4.2:4.2] {2*rad(atan(x))/pi};
\addplot [line width=0.9pt,mark=none,pgreen,
samples=250,domain=-4.2:4.2] {x/(1+abs(x))};
\addplot [line width=0.9pt,mark=none,black,dashed,
samples=250,domain=-0.0:4.2] {1};
\addplot [line width=0.9pt,mark=none,black,dashed,
samples=250,domain=-4.2:0.0] {-1};
\end{axis}
\end{tikzpicture}\\
\end{tabular} 
\caption{Proximal soft clipping operators on $\RR$ with
saturation at $\pm 1$: 
$\eta\mapsto\sign(\eta)(1-\exp(-|\eta|))$
\cite[Section~10.6.3]{Tarr19} (blue),
$\eta\mapsto 2\arctan(\eta)/\pi$ \cite{Svva20} (red), 
and $\eta\mapsto\eta/(1+|\eta|)$ \cite{Marm20} (green).}
\label{fig:sc}
\end{figure}

\subsection{General firmly nonexpansive operators}
\label{sec:33}

Not all firmly nonexpansive operators are proximity operators
\cite{Bord18}. 

\begin{example}
\label{ex:d1}
Let $(R_j)_{1\leq j\leq m}$ be nonexpansive operators on $\GG$. 
Then the operator
\begin{equation}
\label{e:d1}
F=\dfrac{\Id+R_1\circ\cdots\circ R_m}{2}
\end{equation}
is firmly nonexpansive \cite[Proposition~4.4]{Livre1} but it
is not a proximity operator \cite[Example~3.5]{Bord18}.
A concrete instance of \eqref{e:d1} is found in audio signal
processing. Consider a distortion $p\in\GG$ of a linearly degraded
audio signal $L\overline{x}\in\GG$ modeled by 
\begin{equation}
F(L\overline{x})=p,
\end{equation}
where $L$ produces effects such as echo or reverberation
\cite[Chapter~11]{Tarr19}, and $F$ comprises 
several simpler operations $(R_j)_{1\leq j\leq m}$ which are
actually firmly nonexpansive (see, e.g., Example~\ref{ex:k},
\cite{Ibap21}, and \cite[Section~10.6.2]{Tarr19}). 
These simpler distortion operators are then used in series and
blended with a proportion of the input signal
\cite[Section~10.9]{Tarr19}, so that the overall process is
described by \eqref{e:d1} (see Figure~\ref{fig:dist}). 
More generally $F$ remains firmly nonexpansive when
$R_1\circ\cdots\circ R_m$ is replaced by any nonexpansive operator.
\end{example}

\begin{figure}[H]
\centering
\begin{circuitikz} 
\draw
(3.5,0) to[short, -, i=$1/2$] (4.75,0);
\draw[line width=0.28mm]
(0,1) to[] (0.6,1) node[above] {Input}
  to[] (1.25,1)
  to[] (1.25,2)
  to[short, -, i=$1/2$] (4.75,2)
  to[] (4.75,2)
  to[] (4.75,1) node[draw,circle,fill=white,minimum
  size=8pt, inner sep=1pt](p){$+$}
  to[] (5,1)
  to[] (5.8,1) node[above] {Output}
  to[] (6.5,1)
  to[] (4.75,1)
  to[] (4.75,0)
  to[] (3.1,0) node[draw,fill=white,minimum
  size=8pt, inner sep=4pt](r2){$R_{1}$}
  to[] (1.95,0) node[draw,fill=white,minimum
  size=8pt, inner sep=4pt](r1){$R_{2}$}
  to[] (1.25,0)
  to[] (1.25,1)
;
\draw[-latex,line width=0.25mm] (1.25,0) -- (r1.west);
\draw[-latex,line width=0.25mm] (r1.east) -- (r2.west);
\draw[-latex,line width=0.25mm] (4.75,0) -- (p.south);
\draw[-latex,line width=0.25mm] (4.75,2) -- (p.north);
\end{circuitikz}
\caption{The distortion operator $F$ in Example~\ref{ex:d1} for
$m=2$.}
\label{fig:dist}
\end{figure}
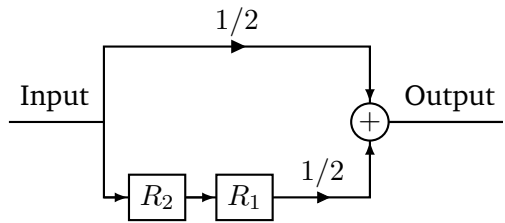

\subsection{Proxification}
\label{sec:34}

In some instances, a prescription $q\in\GG$ may be given by an
equation of the form $Qy=q$, where $Q\colon\GG\to\GG$ is
not firmly nonexpansive. In this section, we provide
constructive examples of \emph{proxification}, by which we mean
the replacement of the equality $Qy=q$ with an equivalent equality
$Fy=p$, where $p\in\GG$ and $F\colon\GG\to\GG$ is firmly
nonexpansive.

\begin{definition}
\label{d:proxif}
Let $Q\colon\GG\to\GG$ and let $q\in\ran Q$. Then $(Q,q)$ is
\emph{proxifiable} if there exists a firmly nonexpansive operator
$F\colon\GG\to\GG$ and $p\in\ran F$ such that 
$(\forall y\in\GG)$ $Qy=q$ $\Leftrightarrow$ $Fy=p$. In this case
$(F,p)$ is a \emph{proxification} of $(Q,q)$.
\end{definition}

We begin with a necessary condition describing when this technique
is possible.

\begin{proposition}
Let $Q\colon\GG\to\GG$ and $q\in\ran Q$ be such that $(Q,q)$ is 
proxifiable. Then
\begin{equation}
\label{e:n1}
Q^{-1}\big(\{q\}\big)=\menge{y\in\GG}{Qy=q}\;\;\text{is closed and
convex}.
\end{equation}
\end{proposition}
\begin{proof}
The proxification assumption means that there exists a firmly 
nonexpansive operator $F\colon\GG\to\GG$ and $p\in\ran F$ such that
$Q^{-1}(\{q\})=F^{-1}(\{p\})$. Now set $T=\Id-F+p$. Then it follows
from \cite[Proposition~4.4]{Livre1} that $T$ is firmly 
nonexpansive, and therefore from \cite[Corollary~4.24]{Livre1} that
$Q^{-1}(\{q\})=F^{-1}(\{p\})=\Fix T$ is closed and convex. 
\end{proof}

Interestingly, condition \eqref{e:n1} is also assumed in various
nonlinear recovery problems \cite{Renc19,Rzep18,Thao94}.
However, the solution techniques of these papers require
the ability to project onto $Q^{-1}(\{q\})$ -- a
capability which rarely occurs when $\dim\GG>1$. The numerical
approach proposed in Section~\ref{sec:4} will circumvent this
requirement and lead to provenly-convergent algorithms which
instead rely on evaluating the associated firmly nonexpansive
operator $F\colon\GG\to\GG$.

\begin{example}[\hspace{1sp}{\cite[Proposition~2.14]{Ibap21}}]
\label{ex:ht}
For every $j\in\{1,\ldots,m\}$, let $\mathsf{G}_j$ be a real
Hilbert space, let $\mathsf{D}_j$ be a nonempty closed convex
subset of $\mathsf{G}_j$, let $\gamma_j\in\RPP$, and set
\begin{equation}
\label{e:ht6}
\mathsf{Q}_j\colon\mathsf{G}_j\to\mathsf{G}_j\colon
\mathsf{y}_j\mapsto \begin{cases}
\mathsf{y}_j,&\text{if}\;\;d_{\mathsf{D}_j}
(\mathsf{y}_j)>\gamma_j;\\
\proj_{\mathsf{D}_j}\mathsf{y}_j,&\text{if}\;\;
d_{\mathsf{D}_j}(\mathsf{y}_j)\leq\gamma_j
\end{cases}
\end{equation}
and
\begin{equation}
\label{e:g3}
\mathsf{S}_j\colon\mathsf{G}_j\to\mathsf{G}_j\colon
\mathsf{y}_j\mapsto
\begin{cases}
\mathsf{y}_j+\dfrac{\gamma_j}{d_{\mathsf{D}_j}(\mathsf{y}_j)}
(\proj_{\mathsf{D}_j}\mathsf{y}_j-\mathsf{y}_j),&\text{if}\;\:
\mathsf{y}_j\not\in\mathsf{D}_j;\\
\mathsf{y}_j,&\text{if}\;\:\mathsf{y}_j\in\mathsf{D}_j.
\end{cases}
\end{equation}
Suppose that $\GG=\bigoplus_{1\leq j\leq m}\mathsf{G}_j$, set
$Q\colon\GG\to\GG\colon(\mathsf{y}_j)_{1\leq j\leq m}\mapsto
(\mathsf{Q}_j\mathsf{y}_j)_{1\leq j\leq m}$, and let $q\in\ran Q$.
Even though $Q$ is discontinuous, $(Q,q)$ is proxifiable. Indeed,
set $S\colon\GG\to\GG\colon(\mathsf{y}_j)_{1\leq j\leq m}\mapsto
(\mathsf{S}_j\mathsf{y}_j)_{1\leq j\leq m}$, 
$F\colon\GG\to\GG\colon(\mathsf{y}_j)_{1\leq j\leq m}\mapsto
(\mathsf{S}_j(\mathsf{Q}_j\mathsf{y}_j))_{1\leq j\leq m}$, 
and $p=Sq$.
Then $(F,p)$ is a proxification of $(Q,q)$. In particular if, for
every $j\in\{1,\ldots,m\}$, $\mathsf{D}_j=\{0\}$, then $Q$ is the
block thresholding estimation operator of
\cite[Section~2.3]{Hall98}.
\end{example}

\begin{example}
\label{ex:ht1}
Consider Example~\ref{ex:ht} with, for every $j\in\{1,\ldots,m\}$,
$\mathsf{G}_j=\RR$, $\mathsf{D}_j=\{0\}$, and 
$\gamma_j=\gamma\in\RPP$. Then each operator $\mathsf{Q}_j$ in 
\eqref{e:ht6} reduces to the hard thresholder 
\begin{equation}
\label{e:ht}
\hard{\gamma}\colon\eta\mapsto
\begin{cases}
\eta,&\text{if}\;\;|\eta|>\gamma;\\
0,&\text{if}\;\;|\eta|\leq\gamma,
\end{cases}
\end{equation}
$\mathsf{S}_j\colon\eta\mapsto\eta-\gamma\sign(\eta)$, and
\begin{equation}
\label{e:5}
\mathsf{S}_j\circ\hard{\gamma}=\soft{\gamma}\colon\eta\mapsto
\sign(\eta)\max\{|\eta|-\gamma,0\}
\end{equation}
is the soft thresholder on $[-\gamma,\gamma]$. Furthermore, it
follows from Example~\ref{ex:ht} that $(F,p)$ is a proxification of
$(Q,q)$. The resulting transformation $Q$ is used for signal
compression in \cite{Dilw03,Teml98}, and as a sensing model in
\cite{Boch13}.
\end{example}

Next, we combine Example~\ref{ex:ht1} with Lemma~\ref{l:m1} to
address low rank matrix approximation. Note the properties of
$\phi$ in Lemma~\ref{l:m1} imply that $\prox_{\phi}0=0$. Therefore,
operators of the form \eqref{e:Fm} cannot increase the rank of a
matrix.

\begin{example}
\label{ex:svd}
Let $\GG$ be the real Hilbert space of $N\times M$ matrices
under the Frobenius norm, set $s=\min\{N,M\}$, and let us denote
the singular value decomposition of $y\in\GG$ by 
$y=U_y\:\diag(\sigma_1(y),\ldots,\sigma_s(y))V_y^\top$. Let
$\rho\in\RPP$, let $\hard{\rho}$ be given by \eqref{e:ht}, set
$\mathsf{S}\colon\RR\to\RR\colon\eta\mapsto\eta-\rho\sign(\eta)$,
and set
\begin{equation}
\label{e:tsvd}
\begin{cases}
Q\colon\GG\to\GG\colon y\mapsto U_y\:
\diag\Big(\hard{\rho}\big(\sigma_1(y)\big),\ldots,\hard{\rho}
\big(\sigma_s(y)\big)\Big)V_y^\top\\[2mm]
S\colon\GG\to\GG\colon y\mapsto U_y\:\diag\Big(
\mathsf{S}\big(\sigma_1(y)\big),\ldots,
\mathsf{S}\big(\sigma_s(y)\big)\Big)V_y^\top.
\end{cases}
\end{equation}
Let $q\in\ran Q$, and set $F=S\circ Q$ and $p=Sq$. Since
$\soft{\rho}=\prox_{\rho|\,\cdot\,|}$ and $\rho|\cdot|$ is
even, it follows from Example~\ref{ex:ht1} and Lemma~\ref{l:m1}
that $(F,p)$ is a proxification of $(Q,q)$. The operator $Q$ 
is used in image compression to produce low rank
approximations \cite{Andr76,Knoc99,Rana07,Yang95}, and the
associated firmly nonexpansive operator $F$ soft-thresholds
singular values at level $\rho$.
\end{example}

\begin{remark}
\label{r:svd}
In the setting of Example~\ref{ex:svd}, consider the
compression technique performed by the nonconvex projection
operator $R\colon\GG\to\GG$ \cite{Cadz88} which truncates singular
values at a given rank $r\in\{1,\ldots,s-1\}$, i.e.,
$R\colon y\mapsto U_y\:\diag\big(\sigma_1(y),\ldots,
\sigma_r(y),0,\ldots,0\big)V_y^\top$.
Let $y\in\GG$ and set $q=Ry$. Then, for every
$\rho\in\left]\sigma_{r+1}(y),\sigma_{r}(y)\right[$, $Qy=q$.
Therefore, knowledge of the low rank approximation $q$ to $y$ 
can be exploited in our framework by proxifying $(Q,q)$ using
Example~\ref{ex:svd}. Note that $\rho$ can be estimated from $q$
since one has access to $\sigma_{r}(q)=\sigma_{r}(y)$.
\end{remark}

Our last example illustrates how proxification can be used to
handle a prescription arising from an extension of the notion of 
a proximity operator for nonconvex functions.

\begin{example}
\label{ex:wc}
Let $\mu\in\RPP$, let $\gamma\in\left]0,1/\mu\right[$, set
$\beta=1-\gamma\mu$, and let $g\colon\GG\to\RX$ be proper, lower
semicontinuous, and \emph{$\mu$-weakly convex} in the sense that
$g+{\mu}\|\cdot\|^2/{2}$ is convex. For every $y\in\GG$, 
$g+\|y-\cdot\|^2/(2\gamma)$ is a strongly convex function in 
$\Gamma_0(\GG)$ and, by \cite[Corollary~11.17]{Livre1}, it 
therefore admits a unique minimizer $Q_{\gamma g}y$, which 
defines the operator $Q_{\gamma g}\colon\GG\to\GG$. Now let 
$q\in\ran Q_{\gamma g}$ and set $A=\partial(g+\mu\|\cdot\|^2/2)$,
$B=A-\mu\Id$, $F=\beta Q_{\gamma g}$, and $p=\beta q$. Then $A$ is
maximally monotone but in general, since $g$ is not convex, 
$Q_{\gamma g}$ is not firmly nonexpansive. However,
\begin{align}
\big(\forall (y,p)\in\GG\times\GG\big)\quad
Q_{\gamma g}y=p
&\Leftrightarrow p\in\zer\Big(\partial\Big(\gamma
g+\dfrac{\gamma\mu}{2}
\|\cdot\|^2-\dfrac{\gamma\mu}{2}\|\cdot\|^2
+\dfrac{1}{2}\|y-\cdot\|^2\Big)\Big)
\nonumber\\
&\Leftrightarrow p\in\zer(\gamma A+\beta\Id-y)
=\zer(\Id+\gamma B-y)\nonumber\\
&\Leftrightarrow J_{\gamma B}y=p,
\end{align}
so Lemma~\ref{l:34} implies that $Q_{\gamma g}=J_{\gamma B}$ is
$\beta$-cocoercive. Thus, $(F,p)$ is a proxification of
$(Q_{\gamma g},q)$. Operators of the form $Q_{\gamma g}$ are used
for shrinkage in \cite{Bayr17,Mali14,Sele17} in the same spirit 
as in Example~\ref{ex:82}. For instance, for $\GG=\RR$ and
$\rho\in\RPP$, the penalty $g=\ln(\rho+|\,\cdot\,|)$ of
\cite{Mali14,Sele17} is $\rho^{-2}$-weakly convex and yields
\begin{equation}
\label{e:lt}
Q_{\gamma g}\colon y\mapsto
\begin{cases}
\dfrac{1}{2}\big(y-\rho +\sqrt{|y+\rho|^2-4\gamma}\,\big),
&\text{if}\;\;y>\dfrac{\gamma}{\rho};\\[3mm]
0,&\text{if}\;\;|y|\leq\dfrac{\gamma}{\rho};\\[2mm]
\dfrac{1}{2}\big(y+\rho-\sqrt{|y-\rho|^2-4\gamma}\,\big),
&\text{if}\;\;y<-\dfrac{\gamma}{\rho}.
\end{cases}
\end{equation}
\end{example}

\subsection{Operators arising from monotone equilibria}
\label{sec:35}

The property that the object of interest is a zero of the sum of
two monotone operators can be modeled in our framework as follows.

\begin{example}
\label{ex:fb}
Let $A\colon\GG\to2^\GG$ be maximally monotone, let
$\beta\in\RPP$, and let $B\colon\GG\to\GG$ be $\beta$-cocoercive.
Let $\gamma\in\left]0,2\beta\right[$ and set
\begin{equation}
\label{e:fb}
F=\bigg(1-\frac{\gamma}{4\beta}\bigg)
\big(\Id-J_{\gamma A}\circ(\Id-\gamma B)\big)
\quad\text{and}\quad p=0.
\end{equation}
Then $F$ is firmly nonexpansive and, for every $y\in\GG$,
$Fy=p\Leftrightarrow y\in\zer(A+B)$. Indeed, set $R=J_{\gamma
A}\circ(\Id-\gamma B)$. By \cite[Proposition~26.1(iv)]{Livre1}, $R$
is $(2-\gamma/2\beta)^{-1}$-averaged and $\zer F=\Fix R=\zer(A+B)$.
It follows from \cite[Proposition~4.39]{Livre1} that $\Id-R$ is
$(1-\gamma/(4\beta))$-cocoercive, which makes $F$ firmly
nonexpansive.
\end{example}

\begin{example}
\label{ex:ms}
Let $f\in\Gamma_0(\GG)$, let $\beta\in\RPP$, and let
$g\colon\GG\to\RR$ be a convex and differentiable function
such that $\nabla g$ is $\beta^{-1}$-Lipschitzian. 
Consider the task of enforcing the property
\begin{equation}
\label{e:bl}
y\in\Argmin (f+g).
\end{equation}
Set $A=\partial f$ and $B=\nabla g$. Then $B$ is $\beta$-cocoercive
\cite[Corollary~18.17]{Livre1}, and \eqref{e:bl} holds if and
only if $y\in\zer(A+B)$. Therefore, Example~\ref{ex:fb} yields a
proximal point representation $(F,p)$ of \eqref{e:bl}.
\end{example}

\section{Analysis and numerical solution of Problem~\ref{prob:2}}
\label{sec:4}

We first show that Problem~\ref{prob:2} is an appropriate
relaxation of Problem~\ref{prob:1}.

\begin{proposition}
\label{p:24}
Suppose that the set of solutions to Problem~\ref{prob:1} is
nonempty. Then it coincides with that of solutions to 
Problem~\ref{prob:2}.
\end{proposition}
\begin{proof}
Let $\overline{x}$ be a solution to Problem~\ref{prob:1}. Then it
is clear that $\overline{x}$ solves Problem~\ref{prob:2}. Now let
$x$ be a solution to Problem~\ref{prob:2}. Then $x\in C$ and 
\begin{equation}
\label{e:241}
(\forall y\in C)\quad\sum_{i\in I}\omega_i
\scal{L_i(x-y)}{F_i(L_ix)-p_i}\leq 0.
\end{equation}
Therefore, since $\overline{x}\in C$ and, for every $i\in I$,
$F_i(L_i\overline{x})=p_i$, we obtain
\begin{equation}
\label{e:242}
\sum_{i\in I}\omega_i
\scal{L_ix-L_i\overline{x}}{F_i(L_ix)-F_i(L_i\overline{x})}\leq 0
\end{equation}
and, by firm nonexpansiveness of the operators $(F_i)_{i\in I}$, 
\begin{equation}
\label{e:243}
\sum_{i\in I}\omega_i\|F_i(L_ix)-F_i(L_i\overline{x})\|^2
\leq\sum_{i\in I}\omega_i
\scal{L_ix-L_i\overline{x}}{F_i(L_ix)-F_i(L_i\overline{x})}\leq 0.
\end{equation}
We conclude that $(\forall i\in I)$ 
$F_i(L_ix)=F_i(L_i\overline{x})=p_i$.
\end{proof}

\begin{remark}
\label{r:22}
Consider the setting of Problem~\ref{prob:2} and set
$\GGG=\bigoplus_{i\in I}\GG_i$, 
$\boldsymbol{L}\colon\HH\to\colon\GGG\colon 
x\mapsto(L_ix)_{i\in I}$, 
$\boldsymbol{F}\colon\GGG\to\GGG\colon (y_i)_{i\in I}\mapsto 
(F_iy_i)_{i\in I}$, and $\boldsymbol{p}=(p_i)_{i\in I}$.
Note that
\begin{equation}
\text{Problem~\ref{prob:1} admits a solution if and only if}\:\:
\boldsymbol{p}\in\boldsymbol{F}\big(\boldsymbol{L}(C)\big).
\end{equation}
Thus, the quantity 
$d_{\boldsymbol{F}(\boldsymbol{L}(C))}(\boldsymbol{p})$ provides a
measure of inconsistency of Problem~\ref{prob:1}. We can actually
use a solution to Problem~\ref{prob:2} to estimate it. Indeed,
suppose that $\overline{x}_1$ and $\overline{x}_2$ are solutions to
\eqref{e:prob2}. Then \eqref{e:f} yields
\begin{align}
\sum_{i\in I}\omega_i
\|F_i(L_i\overline{x}_1)-F_i(L_i\overline{x}_2)\|^2
&\leq\sum_{i\in I}\omega_i
\scal{L_i\overline{x}_1-L_i\overline{x}_2}
{F_i(L_i\overline{x}_1)-F_i(L_i\overline{x}_2)}
\nonumber\\
&=\sum_{i\in I}\omega_i\scal{L_i(\overline{x}_1-\overline{x}_2)}
{F_i(L_i\overline{x}_1)-p_i}
\nonumber\\
&\quad+\sum_{i\in I}\omega_i
\scal{L_i(\overline{x}_2-\overline{x}_1)}
{F_i(L_i\overline{x}_2)-p_i}
\nonumber\\
&\leq 0.
\end{align}
Hence, for every $i\in I$, there exists a unique
$\overline{p}_i\in\GG_i$ such that every solution
$\overline{x}$ to Problem~\ref{prob:2} satisfies
\begin{equation}
\label{e:p}
F_i(L_i\overline{x})=\overline{p}_i.
\end{equation} 
In turn, if $\overline{x}$ is any solution to Problem~\ref{prob:2},
then 
\begin{equation}
\label{e:fgap}
d_{\boldsymbol{F}(\boldsymbol{L}(C))}(\boldsymbol{p})
=\inf_{x\in C}\|\boldsymbol{p}-\boldsymbol{F}(\boldsymbol{L}x)\|
\leq\|\boldsymbol{p}-\boldsymbol{F}(\boldsymbol{L}\overline{x})\|
=\|\boldsymbol{p}-\overline{\boldsymbol{p}}\|= 
\sqrt{\sum_{i\in I}\|p_i-\overline{p}_i\|^2}.
\end{equation}
\end{remark}

Next, we turn to the existence of solutions.

\begin{proposition}
\label{p:8}
Problem~\ref{prob:2} admits a solution in each of the 
following instances.
\begin{enumerate}
\item
\label{p:8i}
$\sum_{i\in I}\omega_iL_i^*p_i\in\ran(N_C+\sum_{i\in I}
\omega_iL_i^*\circ F_i\circ L_i)$.
\item
\label{p:8ii}
$C$ is bounded.
\item
\label{p:8iii}
$\ran N_C+\sum_{i\in I}\omega_i L_i^*(\ran F_i)=\HH$.
\item
\label{p:8iv}
For some $i\in I$, $L_i^*$ is surjective and one of the following
holds:
\begin{enumerate}
\item
\label{p:8iva}
$L_i^*(\ran F_i)=\HH$.
\item
\label{p:8ivb}
$F_i$ is surjective.
\item
\label{p:8ivc}
$\|F_i(y)\|\to\pinf$ as
$\|y\|\to\pinf$.
\item
\label{p:8ivd}
$\ran(\Id-F_i)$ is bounded.
\item
\label{p:8ive}
There exists a continuous convex function $g_i\colon\GG_i\to\RR$
such that $F_i=\prox_{g_i}$.
\end{enumerate}
\end{enumerate}
\end{proposition}
\begin{proof}
Set $A=N_C$ and $(\forall i\in I)$ $B_i=\omega_i F_i$. Then the
operators $(B_i)_{i\in I}$ are cocoercive. Now define
\begin{equation}
\label{e:M}
M=A+\sum_{i\in I}L_i^*\circ B_i\circ L_i.
\end{equation}
It follows from \cite[Proposition~4.12]{Livre1} that 
$B=\sum_{i\in I}L_i^*\circ B_i\circ L_i$ is cocoercive and 
hence maximally monotone by 
\cite[Example~20.31]{Livre1}, with $\dom B=\HH$.
On the other hand, \cite[Example~20.26]{Livre1} asserts
that $A$ is maximally monotone. We therefore derive from
\cite[Corollary~25.5(i)]{Livre1} that 
\begin{equation}
\label{e:za77}
M\;\text{is maximally monotone}.
\end{equation}

\ref{p:8i}: Let $x\in\HH$. In view of \eqref{e:normalcone}, $x$ 
solves Problem~\ref{prob:2} if and only if
\begin{equation}
\label{e:za76}
-\sum_{i\in I}\omega_iL_i^*\big(F_i(L_ix)-p_i\big)\in N_Cx, 
\end{equation}
that is, $\sum_{i\in I}\omega_iL_i^*p_i\in Mx$.

\ref{p:8ii}: Since $\dom M=\dom A=C$ is bounded, it follows 
from \eqref{e:za77} and \cite[Corollary~21.25]{Livre1} that $M$ is
surjective, so \ref{p:8i} holds. 

\ref{p:8iii}: It follows from \cite[Example~25.14]{Livre1} that
$A$ is $3^*$ monotone and from \cite[Example~25.20(i)]{Livre1} 
that the operators $(B_i)_{i\in I}$ are likewise. Hence, in view 
of \eqref{e:za77} we invoke Lemma~\ref{l:2} to get
\begin{equation}
\inte\ran M=
\inte\ran\bigg(A+\sum_{i\in I}L_i^*\circ B_i\circ L_i\bigg)=
\inte\bigg(\ran A+\sum_{i\in I}L_i^*(\ran B_i)\bigg)=
\HH.
\end{equation}
So $M$ is surjective and \ref{p:8i} holds. 

\ref{p:8ivb}$\Rightarrow$\ref{p:8iva}$\Rightarrow$\ref{p:8iii}:
Clear.

\ref{p:8ivc}$\Rightarrow$\ref{p:8ivb}:
Since $F_i$ is maximally monotone by 
\cite[Example~20.30]{Livre1}, this follows from 
\cite[Corollary~21.24]{Livre1}.

\ref{p:8ivd}$\Rightarrow$\ref{p:8ivc}: Set
$\rho=\sup_{y\in\GG_i}\|y-F_iy\|$. Then 
$\|F_iy\|\geq\|y\|-\|y-F_iy\|\geq\|y\|-\rho\to\pinf$ as
$\|y\|\to\pinf$.

\ref{p:8ive}$\Rightarrow$\ref{p:8ivb}: 
We derive from \cite[Proposition~16.27]{Livre1} that
$\GG_i=\intdom g_i\subset\dom\partial g_i=\dom(\Id+\partial g_i)
=\ran(\Id+\partial g_i)^{-1}=\ran\prox_{g_i}$.
\end{proof}

\begin{example}
\label{ex:34}
A simple instance when Problem~\ref{prob:1} has no solution, while
the relaxed Problem~\ref{prob:2} does, is the following. Take 
disjoint nonempty closed convex subsets $C$ and $D$ of $\HH$
such that $C$ is bounded, and let $I={1}$, $\GG_1=\HH$,
$L_1=\Id$, $F_1=\Id-\proj_{D}$, and $p_1=0$. Then the solution 
set of Problem~\ref{prob:1} is $C\cap D=\emp$, while that 
of Problem~\ref{prob:2} is
$\Fix(\proj_C\circ\proj_D)\neq\emp$ \cite{Gubi67}.
\end{example}

We have described in Example~\ref{ex:1} an instance of the relaxed
Problem~\ref{prob:2} which is in fact a minimization problem. The
next proposition describes a general setting in which a
minimization problem underlies Problem~\ref{prob:2}. It involves
the Moreau envelope of \eqref{e:18}. 

\begin{proposition}
\label{p:9}
Consider the setting of Problem~\ref{prob:2} and suppose that, 
for every $i\in I$, there exists $g_i\in\Gamma_0(\GG_i)$ such that
$F_i=\prox_{g_i}$. Then the objective of Problem~\ref{prob:2} 
is to
\begin{equation}
\label{e:prob7}
\minimize{x\in C}{f(x)},\quad\text{where}\quad 
f\colon x\mapsto\frac{1}{2}\sum_{i\in I}\,\omega_i
\Big(\widetilde{g_i^*}(L_ix)-\scal{L_ix}{p_i}\Big).
\end{equation}
\end{proposition}
\begin{proof}
We derive from \cite[Proposition~24.4]{Livre1} that 
$(\forall i\in I)$ $\nabla\widetilde{g_i^*}=\prox_{g_i}$.
In turn, $f$ is differentiable and 
\begin{equation}
(\forall x\in\HH)\quad\nabla f(x)=\sum_{i\in I}\,\omega_i
L_i^*\big(\prox_{g_i}(L_ix)-p_i\big)=\sum_{i\in I}
\,\omega_iL_i^*\big(F_i(L_ix)-p_i\big).
\end{equation}
Consequently, \eqref{e:prob2} is equivalent to finding a solution
to \eqref{e:15}, i.e., by Lemma~\ref{l:1}, to minimizing $f$ over 
$C$.
\end{proof}

Next, we present a block-iterative algorithm for solving 
Problem~\ref{prob:2}. 

\begin{proposition}
\label{p:2}
Consider the setting of Problem~\ref{prob:2} under the assumption
that it has a solution. Let $K$ be a strictly positive integer and
let $(I_n)_{n\in\NN}$ be a sequence of nonempty subsets of $I$ such
that 
\begin{equation}
\label{e:K}
(\forall n\in\NN)\quad\bigcup_{k=0}^{K-1}I_{n+k}=I.
\end{equation}
Let $x_0\in\HH$, let $\gamma\in\left]0,2\right[$, and, for every
$i\in I$, let $t_{i,-1}\in\HH$ and set $\gamma_i=\gamma/\|L_i\|^2$.
Iterate
\begin{equation}
\label{e:a3}
\begin{array}{l}
\text{for}\;n=0,1,\ldots\\
\left\lfloor
\begin{array}{l}
\text{for every}\;i\in I_n\\
\left\lfloor
\begin{array}{l}
t_{i,n}=x_n-\gamma_iL_i^*\big(F_i(L_ix_n)-p_i\big)\\
\end{array}
\right.\\
\text{for every}\;i\in I\smallsetminus I_n\\
\left\lfloor
\begin{array}{l}
t_{i,n}=t_{i,n-1}\\
\end{array}
\right.\\[1mm]
x_{n+1}=\proj_{C}\Bigg(\Sum_{i=1}^m\omega_it_{i,n}\Bigg).
\end{array}
\right.\\
\end{array}
\end{equation}
Then $(x_n)_{n\in\NN}$ converges weakly to a solution to
Problem~\ref{prob:2}.
\end{proposition}
\begin{proof}
Set $A_0=N_C$ and $(\forall i\in I)$ 
$A_i=\|L_i\|^{-2}(L_i^*\circ F_i\circ L_i-L_i^*p_i)$.
For every $i\in I$, since $F_i$ is firmly nonexpansive, it follows
from \cite[Proposition~4.12]{Livre1} that $A_i$ is firmly
nonexpansive, i.e., cocoercive with $\beta_i=1$. Thus, \eqref{e:a3}
is a special case of \eqref{e:a30}, and the conclusion follows from
Proposition~\ref{p:20}.
\end{proof}

An attractive feature of \eqref{e:a3} is its ability to activate
only a subblock of operators $(F_i)_{i\in I_n}$ at iteration $n$,
as opposed to all of them as in classical algorithms dealing with
inconsistent common fixed point problems
\cite{Cens05,Cens18,Sign94,Siop13}. This flexibility is of the
utmost relevance for very large-scale applications. It will also be
seen in Section~\ref{sec:5} to lead to more efficient
implementations. Condition \eqref{e:K} regulates the frequency of
activation of the operators. Since $K$ can be chosen arbitrarily,
it is actually quite mild.

\section{Numerical experiments}
\label{sec:5}

In this section, we illustrate the ability of the proposed
framework to efficiently model and solve various signal and image
recovery problems with inconsistent nonlinear prescriptions. Each
instance will use the block-iterative algorithm~\eqref{e:a3} which
was shown in Proposition~\ref{p:2} to produce an exact solution of
Problem~\ref{prob:2} from any initial point in $\HH$. Here, we
implement it with $x_0=0$.

\begin{remark}
\label{rmk:pro}
In the modeling of signal construction problems as minimization
problems, it is common practice to add a function $g$ to the
objective in order to promote desirable properties in the
solutions. Several functions are thus averaged and contribute
collectively to defining solutions. A prominent example is the
promotion of sparsity through the addition of a penalty such as the
$\ell^1$ norm in $\RR^N$ \cite{Cand06,Tibs11}. In the more general
variational inequality setting of Problem~\ref{prob:2}, this
template can be mimicked by adding the prescription $Fy=0$, where
$F=\Id-\prox_g$, i.e., by Moreau's decomposition,
$F=\prox_{g^*}$ \cite[Remark~14.4]{Livre1}. Note that exact
satisfaction of the equality $Fy=0$ would just mean that one
minimizes $g$ since $\Fix\prox_g=\Argmin g$. In general, when
incorporated to Problem~\ref{prob:2}, the pair
$(F,p)=(\Id-\prox_g,0)$ is intended to promote the properties $g$
would in a standard minimization problem. We investigate in
Sections~\ref{sec:43} and \ref{sec:44} this technique to encourage
sparsity in $\RR^N$ through the incorporation of the operator
$F=\proj_{B_\infty(0;\rho)}=\Id-\prox_{\rho\|\cdot\|_1}$, where
$B_\infty(0;\rho)$ is the $\ell^\infty$ ball of $\RR^N$ centered at
the origin and with radius $\rho\in\RPP$.
\end{remark}

\subsection{Image recovery}
\label{sec:41}
The goal is to recover the original image
$\overline{x}\in\HH=\RR^N$ ($N=256^2$) shown in
Figure~\ref{fig:3}(a)
from the following.
\begin{itemize}
\item
Bounds on pixel values: $\overline{x}\in C=[0,255]^N$.
\item
The degraded image $p_1\in\GG_1=\HH$ shown in
Figure~\ref{fig:3}(b), which is modeled as follows. The image
$\overline{x}$ is blurred by $L_1\colon\HH\to\GG_1$, which performs
discrete convolution with a $15\times 15$ Gaussian kernel with
standard deviation of $3.5$, then corrupted by an additive noise
$w_1\in\GG_1$. The blurred image-to-noise ratio is
$20\log_{10}(\|L_1\overline{x}\|/\|w_1\|)=24.0$ dB. Pixel values
beyond $60$ are then clipped. Altogether,
$p_1=\proj_{D_1}(L_1\overline{x}+w_1)$, where $D_1=[0,60]^N$. This
process models a low-quality image acquired by a device which
cannot detect photon counts beyond a certain threshold. We
therefore use $F_1=\proj_{D_1}$ in \eqref{e:prob2}.
\item
An approximation of the mean pixel value $\rho_2=138$ of
$\overline{x}$. To enforce this information, following
Example~\ref{ex:1}, we set $\GG_2=\HH$, $L_2=\Id$, $p_2=0$, and 
\begin{equation}
F_2\colon (\eta_k)_{1\leq k\leq N}\mapsto
x-\bigg(\dfrac{\sum_{k=1}^N\eta_k}{N}-\rho_2\bigg)\boldsymbol{1}.
\end{equation}
\item
The phase $\theta\in[-\pi,\pi]^N$ of the 2-D discrete
Fourier transform of a noise-corrupted
version of $\overline{x}$, i.e., 
$\theta=\angle\dft(\overline{x}+w_3)$, where $w_3\in\HH$ 
yields an image-to-noise
ratio $20\log_{10}(\|\overline{x}\|/\|w_3\|)=49.0$~dB. To model
this information, we set $\GG_3=\HH$, $L_3=\Id$, $p_3=0$, and
\begin{equation}
F_3\colon y\mapsto y-\idft\bigg(\big|\dft{y}\big|
\max\Big\{\cos\big(\angle(\dft{y})-\theta\big),0\Big\}
\exp(i\theta)\bigg).
\end{equation}
\end{itemize}
Due to the noise present in $p_1$ and $\theta$, and the inexact
estimation of $\rho_2$, this instance of Problem~\ref{prob:1}
($I=\{1,2,3\}$) is inconsistent. We thus arrive at the relaxed
Problem~\ref{prob:2} by setting $\omega_1=\omega_2=\omega_3=1/3$.
By Proposition~\ref{p:8}\ref{p:8ii}, since $C$ is bounded,
Problem~\ref{prob:2} is guaranteed to possess a solution. The
solution shown in Figure~\ref{fig:3}(c) is computed using
algorithm~\eqref{e:a3} with $\gamma=1.9$ and 
$(\forall n\in\NN)$ $I_n=I$. This experiment illustrates a
nonlinear recovery scenario with inconsistent measurements which
nonetheless produces realistic solutions obtained by
exploiting all available information. 

\begin{figure}[H]
\centering
\begin{tabular}{@{}c@{}c@{}c@{}}
\includegraphics[width=5.25cm]{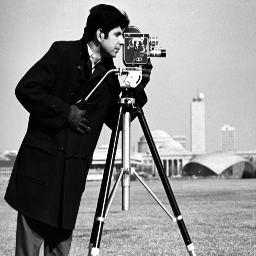}&
\hspace{0.2cm}
\includegraphics[width=5.25cm]{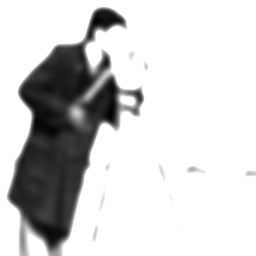}&
\hspace{0.2cm}
\includegraphics[width=5.25cm]{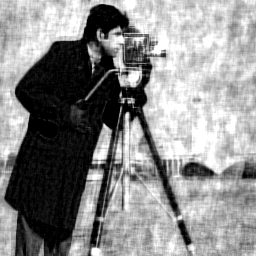}\\
\small{(a)} & \small{(b)} & \small{(c)}
\end{tabular} 
\caption{
Experiment of Section~\ref{sec:41}:
(a) Original image $\overline{x}$.
(b) Degraded image $p_1$.
(c) Recovered image.
}
\label{fig:3}
\end{figure}

\subsection{Signal recovery}
\label{sec:42}
The goal is to recover the original signal $\overline{x}\in\HH=C=
\RR^N$ ($N=1024$) shown in Figure~\ref{fig:1d}(a) from the
following.
\begin{itemize}
\item
A piecewise constant approximation $p_1$ of $\overline{x}$,
given by $p_1=\proj_{D_1}(\overline{x}+w_1)$, where $w_1\in\GG_1$
represents noise and $D_1$ is the subspace of signals in
$\GG_1=\HH$ which are constant by blocks along each of the $16$
sets of $64$ consecutive indices in $\{1,\ldots,N\}$ (see
Figure~\ref{fig:1d}(b)). The signal-to-noise ratio is
$20\log_{10}(\|\overline{x}\|/\|w_1\|)=-2.3$~dB. We model this
observation by setting $L_1=\Id$ and $F_1=\proj_{D_1}$.
\item
A bound $\rho_2=0.025$ on the magnitude of the finite differences
of $\overline{x}$. To enforce this information, following
Example~\ref{ex:1}, we set $\GG_2=\RR^{N-1}$,
$L_2\colon\HH\to\GG_2\colon (\xi_i)_{1\leq i\leq N}\mapsto
(\xi_{i+1}-\xi_i)_{1\leq i\leq N-1}$, $p_2=0$, and
$F_2=\Id-\proj_{D_2}$, where
$D_2=\menge{y\in\GG_2}{\|y\|_{\infty}\leq\rho_2}$, that is,
using \eqref{e:5},
\begin{equation}
F_2\colon (\eta_k)_{1\leq k\leq N-1}\mapsto
\big(\soft{\gamma}(\eta_k)\big)_{1\leq k\leq N-1}.
\end{equation}
\item
A collection of $m=1200$ noisy thresholded scalar observations
$r_3=(\chi_j)_{j\in J}\in\RR^{m}$ of $\overline{x}$, where
$J=\{3,\ldots,m+2\}$. The true data formation model is
\begin{equation}
\label{e:qk}
(\forall j\in J)\quad
\chi_j=R(\scal{\overline{x}}{e_j})+\nu_j,
\end{equation}
where $(e_j)_{j\in J}$ is a dictionary of random vectors in
$\mathbb{R}^N$ with zero-mean i.i.d.\ entries, the noise vector
$w_3=(\nu_j)_{j\in J}$ yields a signal-to-noise ratio of
$20\log_{10}(\|r_3\|/\|w_3\|)=17.8$~dB, and $R$ is the
thresholding operator of the type found in \cite{Abra98,Taov00}
($\rho=0.05$), namely
\begin{equation}
\label{e:22w}
R\colon\RR\to\RR\colon\eta\mapsto
\begin{cases}
\sign(\eta)\sqrt[4]{\eta^4-\rho^4},&\text{if}\;\;|\eta|>\rho;\\
0,&\text{if}\;\;|\eta|\leq\rho.
\end{cases} 
\end{equation}
We assume that $R$ is misspecified and that the presence of noise
is unknown, so that the data acquisition process is incorrectly
modeled as
\begin{equation}
\label{e:qb}
(\forall j\in J)\quad \chi_j=Q(\scal{\overline{x}}{e_j}),
\end{equation}
where
\begin{equation}
\label{e:22-1}
Q\colon\RR\to\RR\colon\eta\mapsto
\begin{cases}
\sign(\eta)\sqrt{\eta^2-\rho^2},&\text{if}\;\;|\eta|>\rho;\\
0,&\text{if}\;\;|\eta|\leq\rho.
\end{cases} 
\end{equation}
Note that $Q$ is not Lipschitzian. Nonetheless, with
\begin{equation}
\label{e:ex2-p}
S\colon\RR\to\RR\colon\eta\mapsto
\sign(\eta)\left(\sqrt{\eta^2+\rho^2}-\rho\right),
\end{equation}
it is straightforward to verify that $S\circ Q=\soft{\rho}$ and
that, for every $j\in J$, $(F_j,p_j)=(\soft{\rho},S\chi_j)$ is a
proxification of $(Q,\chi_j)$. Also, for every $j\in J$, set
$\GG_j=\RR$ and $L_j=\scal{\cdot}{e_j}$.
\end{itemize}
We thus obtain an instantiation of Problem~\ref{prob:2} with
$I=\{1,2\}\cup J$ and, for every $i\in I$, $\omega_i=1/(\card\,I)$.
Since $(e_j)_{j\in J}$ is overcomplete and,
for every $j\in J$, $F_j$ is surjective, it follows that
$\HH=\menge{\sum_{j\in J}\omega_j\eta_je_j}{\eta_j\in\ran F_j}
=\sum_{j\in J}\omega_jL_j^*(\ran F_j)\subset\sum_{i\in I}
\omega_iL_i^*(\ran F_i)$, so Problem~\ref{prob:2} is guaranteed
to possess a solution by Proposition~\ref{p:8}\ref{p:8iii}. 
Algorithm~\eqref{e:a3} produces the signal shown in
Figure~\ref{fig:1d}(c) with $\gamma=1.9$ and the following
activation strategy. At every iteration, $F_1$ and $F_2$ are
activated, while we partition $J$ into four blocks of $300$
elements and cyclically activate one block per iteration, i.e.,
\begin{equation}
\label{e:cycle}
(\forall n\in\NN)(\forall j\in\{0,1,2,3\})\quad
I_{4n+j}=\{1,2,3+300j,\ldots,2+300(j+1)\},
\end{equation}
which satisfies condition~\eqref{e:K} with $K=4$. 
This shows that, even when the data is noisy and
poorly modeled, Problem~\ref{prob:2} produces quite robust
recoveries. The execution time savings resulting from the use of
\eqref{e:cycle} compared to the full activation strategy
(i.e., $I_n=I$ for every $n\in\NN$) are displayed in
Figure~\ref{fig:1derr}. Note that in very large-scale scenarios
in which all data cannot be simultaneously loaded into memory,
activation strategies such as \eqref{e:cycle} make
algorithm~\eqref{e:a3} implementable. 

\begin{figure}[H]
\centering
\begin{tikzpicture}[scale=1.0]
\begin{axis}[height=7.4cm,width=15.5cm, legend cell align={left},
xmin =0, xmax=1024, ymin=-0.6, ymax=1.1]
\addplot [thick, mark=none, color=bluep] table[x={n}, y={orig}]
{figures/ex2/orig.txt};
\end{axis}
\end{tikzpicture}\\
(a)\\[0.3em]
\begin{tikzpicture}[scale=1.0]
\begin{axis}[height=7.4cm,width=15.5cm, legend cell align={left},
xmin =0, xmax=1024, ymin=-0.6, ymax=1.1]
\addplot [thick, mark=none, color=bluep] table[x={n}, y={constant}]
{figures/ex2/constant.txt};
\end{axis}
\end{tikzpicture}\\
(b)\\[0.3em]
\begin{tikzpicture}[scale=1.0]
\begin{axis}[height=7.4cm,width=15.5cm, legend cell align={left},
xmin =0, xmax=1024, ymin=-0.6, ymax=1.1]
\addplot [thick, mark=none, color=bluep] table[x={n}, y={recovered}]
{figures/ex2/recovered.txt};
\end{axis}
\end{tikzpicture}\\
(c)
\caption{Experiment of Section~\ref{sec:42}: (a): Original signal
$\overline{x}$. (b): Piecewise constant approximation $p_1$. (c):
Recovered signal.}
\label{fig:1d}
\end{figure}

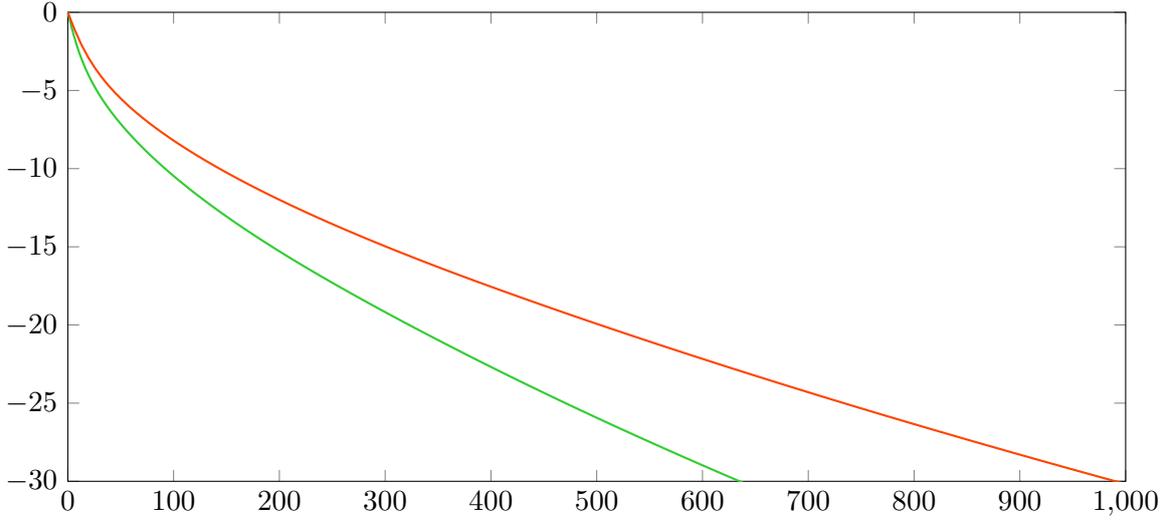
\begin{figure}[H]
\centering
\begin{tikzpicture}[scale=1.0]
\begin{axis}[height=7.8cm,width=15.5cm, legend cell align={left},
xmin =0, xmax=1000, ymin=-30, ymax=0.0]
\addplot [thick, mark=none, color=pgreen] table[x={time}, y={err}]
{figures/ex2/cycledb.txt};
\addplot [thick, mark=none, color=pred] table[x={time}, y={err}]
{figures/ex2/fulldb.txt};
\end{axis}
\end{tikzpicture}
\caption{Experiment of Section~\ref{sec:42}: Relative error
$20\log_{10}(\|x_n-x_{\infty}\|/\|x_0-x_{\infty}\|)$ (dB)
versus execution time (seconds) for full activation
(red) and cyclic activation \eqref{e:cycle}
(green).}
\label{fig:1derr}
\end{figure}

\subsection{Sparse image recovery}
\label{sec:43}
The goal is to recover the original image
$\overline{x}\in\HH=\RR^N$ ($N=256^2$) shown in
Figure~\ref{fig:t1}(a) from the following.
\begin{itemize}
\item
Bounds on pixel values: $x\in C=[0,255]^N$.
\item
The low rank approximation $q_1\in\GG_1=\HH$ displayed in 
Figure~\ref{fig:t1}(b) of a blurred noisy version of
$\overline{x}$ modeled as follows. The blurring operator
$L_1\colon\HH\to\GG_1$ applies a discrete convolution with a
uniform $7\times 7$ kernel, and the operators $Q$ and $S$ are as in
Example~\ref{ex:svd}, with threshold $\rho=500$. Then
$q_1=Q(L_1\overline{x}+w_1)$ is a rank-$85$ compression, where
$w_1\in\GG_1$ induces a blurred image-to-noise ratio of
$20\log_{10}(\|L_1\overline{x}\|/\|w_1\|)=17.6$~dB. By
Example~\ref{ex:svd}, we obtain a proxification of $(Q,q_1)$ with
$(F_1,p_1)=(S\circ Q,Sq_1)$.
\item
$\overline{x}$ is sparse. To promote this property in the 
solutions to \eqref{e:prob2}, following Remark~\ref{rmk:pro}, we
set $\GG_2=\HH$, $L_2=\Id$, $p_2=0$, $\rho_2=1.5$, and
$F_2=\proj_{B_\infty(0;\rho_2)}$.
\end{itemize}
We therefore arrive at an instance of Problem~\ref{prob:2}
with $I=\{1,2\}$ and $\omega_1=\omega_2=1/2$. Since $C$ is bounded,
Proposition~\ref{p:8}\ref{p:8ii} guarantees that a solution exists.
Algorithm~\eqref{e:a3} with $\gamma=1$ yields the recovery in
Figure~\ref{fig:t1}(c). Even though computing $F_1$ requires only
one singular value decomposition (not two, as \eqref{e:tsvd} may
suggest), it is the most numerically expensive operator in this
problem. Therefore, we choose to activate $F_1$ only every $5$
iterations, i.e.,
\begin{equation}
\label{e:skip}
I_n=\begin{cases}
I\smallsetminus\{1\},&\text{if}\;\; n\not\equiv 0\mod 5;\\
I,&\text{if}\;\;n\equiv 0\mod 5.
\end{cases}
\end{equation}
Figure~\ref{fig:terr} displays the time savings resulting from 
the use of \eqref{e:skip} compared to full activation (both
activation strategies yield visually indistinguishable
recoveries). Notice that, while the observation in
Figure~\ref{fig:t1}(b) is virtually illegible, many of the words in
the recovery of Figure~\ref{fig:t1}(c) can be discerned. 

\begin{figure}[H]
\centering
\begin{tabular}{@{}c@{}c@{}c@{}}
\includegraphics[width=5.25cm]{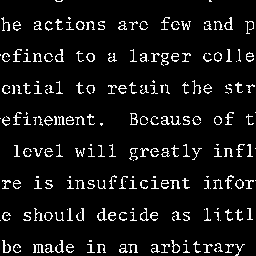}&
\hspace{0.2cm}
\includegraphics[width=5.25cm]{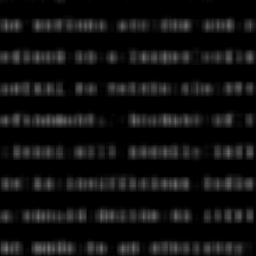}&
\hspace{0.2cm}
\includegraphics[width=5.25cm]{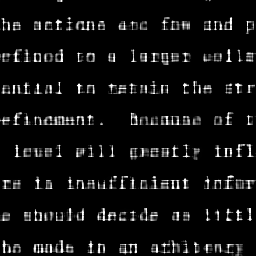}\\
\small{(a)} & \small{(b)} & \small{(c)}
\end{tabular} 
\caption{
Experiment of Section~\ref{sec:43}:
(a) Original image $\overline{x}$.
(b) Degraded image $q_1$.
(c) Recovered image.
}
\label{fig:t1}
\end{figure}

Finally, we examine the use of the non firmly nonexpansive
sparsity-promoting operator of Example~\ref{ex:wc}. Specifically,
$Q_{\gamma g}$ is given by \eqref{e:lt}, which is induced by the
logarithmic penalty with parameters $\rho=\rho_2$ and
$\gamma=0.05/\rho_2^2$. This implies that $0.95 Q_{\gamma g}$ is
firmly nonexpansive and hence that $\Id-0.95 Q_{\gamma g}$ is
likewise. Figure~\ref{fig:t2} displays the result when $F_2$ is
replaced by componentwise applications of $\Id-0.95 Q_{\gamma g}$.
In this experiment, the $\ell^1$ penalty-based operator $F_2$
yields a sharper recovery in Figure~\ref{fig:t1}(c) than the
recovery in Figure~\ref{fig:t2}, which is induced by the
logarithmic penalty.

\begin{figure}[b]
\centering
\includegraphics[width=5.25cm]{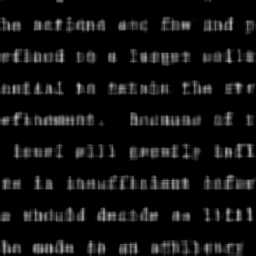}\\
\caption{
Experiment of Section~\ref{sec:43}: Recovered image with
logarithmic thresholding instead of soft thresholding.
}
\label{fig:t2}
\end{figure}

\begin{figure}[H]
\centering
\begin{tikzpicture}[scale=1.0]
\begin{axis}[height=7.8cm,width=15.5cm, legend cell align={left},
xmin =0, xmax=1700, ymin=-30, ymax=0.0]
\addplot [thick, mark=none, color=pgreen] table[x={time}, y={err}]
{figures/ex3/skip5db.txt};
\addplot [thick, mark=none, color=pred] table[x={time}, y={err}]
{figures/ex3/fulldb.txt};
\end{axis}
\end{tikzpicture}
\caption{Experiment of Section~\ref{sec:43}: Relative error
$20\log_{10}(\|x_n-x_{\infty}\|/\|x_0-x_{\infty}\|)$ (dB)
versus execution time (seconds) for full-activation
(red) and block activation \eqref{e:skip} (green).}
\label{fig:terr}
\end{figure}
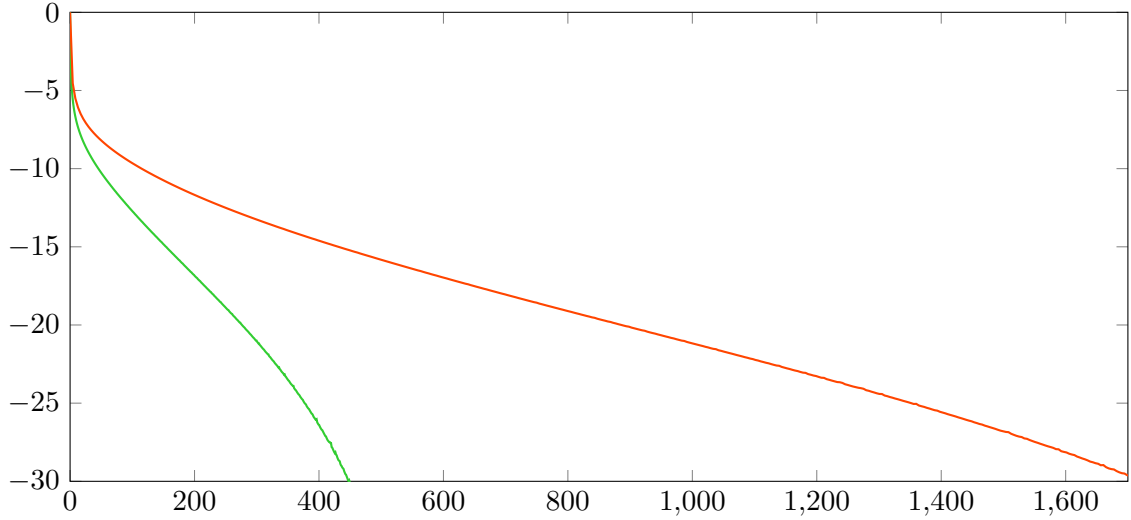

\subsection{Source separation}
\label{sec:44}
This experiment incorporates nonlinear compression to a problem in
astronomy, which seeks to separate a background image
$\overline{x}_1\in\RR^N$ ($N=600^2$) of stars from a galaxy image
$\overline{x}_2\in\RR^N$ \cite{McCo14}. The goal is to
construct the image pair
$(\overline{x}_1,\overline{x}_2)\in\HH=\RR^N\times\RR^N$ given the
following.
\begin{itemize}
\item
Bounds on pixel values: $(\overline{x}_1,\overline{x}_2)\in
C=[0,255]^N\times [0,255]^N$.
\item
The low rank approximation $q_1\in\GG_1=\RR^N$ shown in 
Figure~\ref{fig:s1}(b) of the original superposition 
$\overline{x}_1+\overline{x}_2$ shown in 
Figure~\ref{fig:s1}(a), which is modeled as follows. Set
$L_1\colon\HH\to\GG_1\colon(x_1,x_2)\mapsto x_1+x_2$, and let $Q$
and $S$ be as in Example~\ref{ex:svd} with $\rho=1500$. The
resulting rank-$22$ approximation of
$\overline{x}_1+\overline{x}_2$ is given by
$q_1=Q(L_1(\overline{x}_1,\overline{x}_2))$. It follows from
Example~\ref{ex:svd} that $(F_1,p_1)=(S\circ Q,Sq_1)$ is a
proxification of $(Q,q_1)$.
\item
$\overline{x}_1$ is sparse, and $\overline{x}_2$ admits a sparse
representation relative to the $2$-D discrete cosine transform
$L\colon\RR^N\to\RR^N$ \cite{McCo14}. To encourage these
properties, as discussed in Remark~\ref{rmk:pro}, we set
$\GG_2=\HH$, $L_2\colon (x_1,x_2)\mapsto (x_1,Lx_2)$, $p_2=0$, and
$F_2\colon(y_1,y_2)\mapsto
(\proj_{B_\infty(0;10)}y_1,\proj_{B_\infty(0;45)}y_2)$. In view of
Example~\ref{ex:hc}, $F_2$ is firmly nonexpansive.
\end{itemize}
Thus, we arrive at an instance of Problem~\ref{prob:2} with
$I=\{1,2\}$ and $\omega_1=\omega_2=1/2$. By
Proposition~\ref{p:8}\ref{p:8ii} this problem is guaranteed to
possess a solution, since $C$ is bounded. Algorithm~\eqref{e:a3}
with $\gamma=1$ provides the solution shown in
Figure~\ref{fig:s1}(c)--(d). To improve algorithmic performance,
we adopt the activation strategy \eqref{e:skip}; see
Figure~\ref{fig:serr} for time savings compared to the
full activation strategy. As can be seen from
Figure~\ref{fig:s1}, this approach produces effective recoveries.
Even though this problem involves a discontinuous observation
process, we can nonetheless solve it with algorithm~\eqref{e:a3},
which exploits all of the information at hand.

\begin{figure}[t]
\centering
\begin{tabular}{@{}c@{}c@{}}
\includegraphics[width=5.25cm]{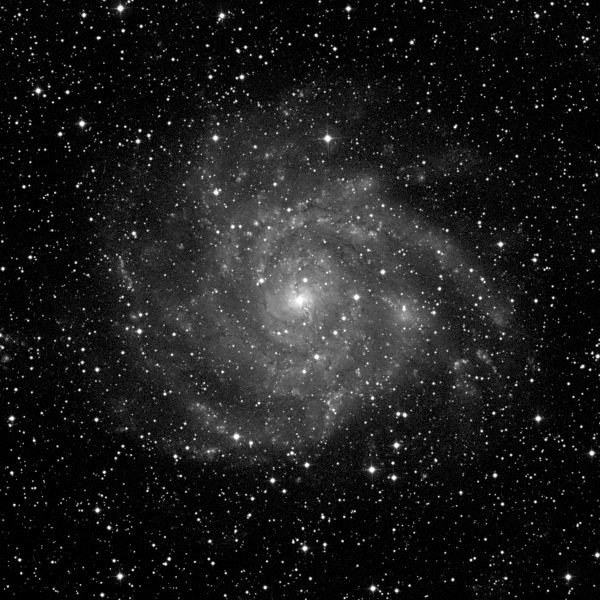}&
\hspace{0.2cm}
\includegraphics[width=5.25cm]{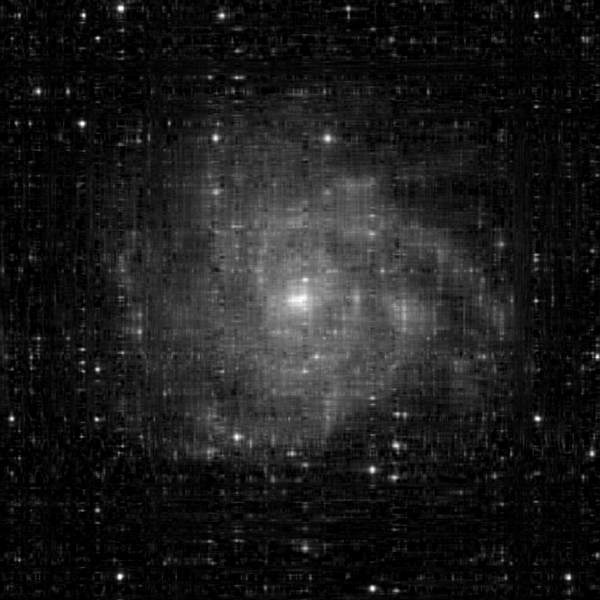}\\
\small{(a)} & \small{(b)}\\
\includegraphics[width=5.25cm]{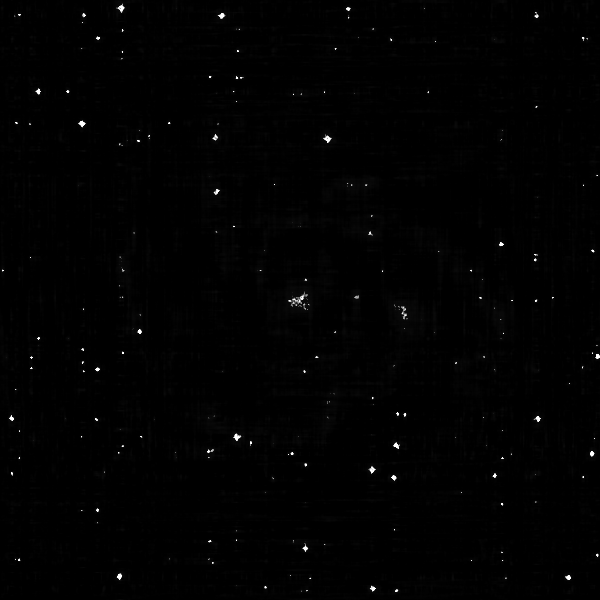}&
\hspace{0.2cm}
\includegraphics[width=5.25cm]{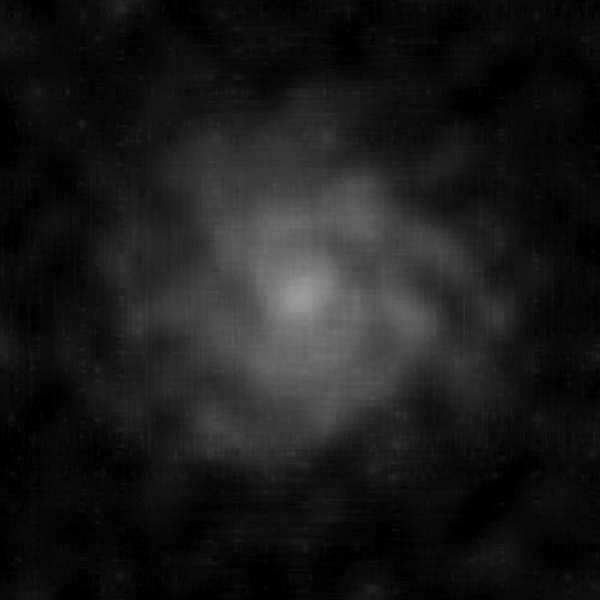}\\
\small{(c)} & \small{(d)}\\
\end{tabular} 
\caption{
Experiment of Section~\ref{sec:44}:
(a) Original image $\overline{x}_1+\overline{x}_2$.
(b) Low-rank compression of $\overline{x}_1+\overline{x}_2$.
(c) Recovered background (stars).
(d) Recovered foreground (galaxy).
}
\label{fig:s1}
\end{figure}

\begin{figure}[H]
\centering
\begin{tikzpicture}[scale=1.0]
\begin{axis}[height=7.8cm,width=15.5cm, legend cell align={left},
xmin =0, xmax=900, ymin=-30, ymax=0.0]
\addplot [thick, mark=none, color=pgreen] table[x={time}, y={err}]
{figures/ex4/skip5db.txt};
\addplot [thick, mark=none, color=pred] table[x={time}, y={err}]
{figures/ex4/fulldb.txt};
\end{axis}
\end{tikzpicture}
\caption{Experiment of Section~\ref{sec:44}: Relative error
$20\log_{10}(\|x_n-x_{\infty}\|/\|x_0-x_{\infty}\|)$ (dB)
versus execution time (seconds) for full-activation
(red) and block activation \eqref{e:skip} (green).}
\label{fig:serr}
\end{figure}
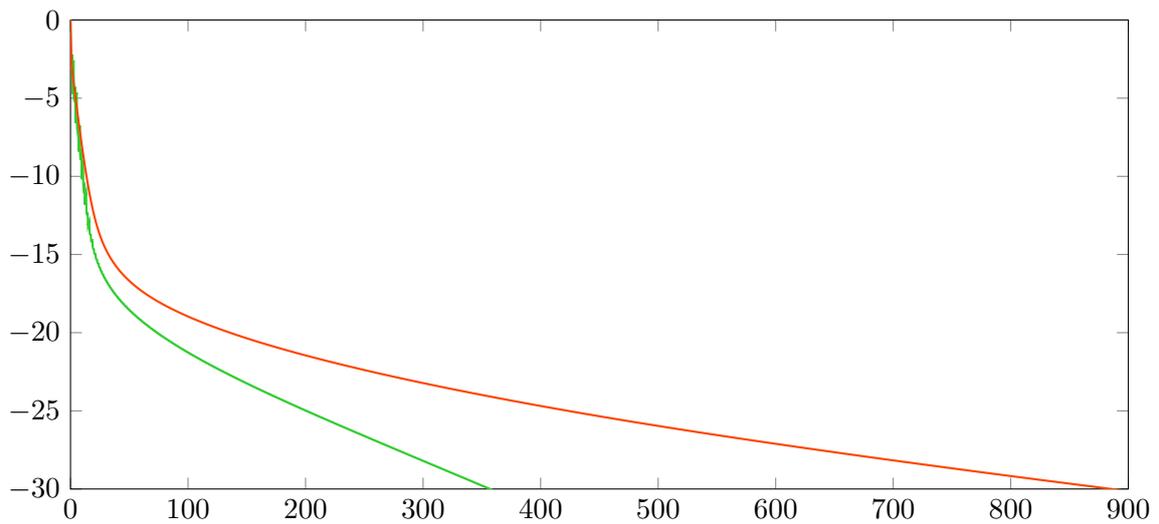

\end{document}